\documentclass[11pt]{article}
\usepackage{amsfonts}
\usepackage{graphics}
\usepackage{indentfirst}
\usepackage{color}
\usepackage{cite}
\usepackage{latexsym}
\usepackage[paper=a4paper, left=1.6cm, right=1.6cm, top=1.8cm, bottom=1.6cm, headheight=5.5pt, footskip=0.8cm, footnotesep=0.8cm, centering, includefoot]{geometry}
\usepackage{amsmath}
\allowdisplaybreaks
\usepackage{amssymb}
\usepackage[colorlinks, linkcolor=red]{hyperref}
\hypersetup{urlcolor=red, citecolor=red}
\usepackage[dvips]{epsfig}
\usepackage{amscd}

\usepackage{amsthm}
\usepackage{mathrsfs}
\usepackage{verbatim}
\newtheorem{theorem}{Theorem}[section]
\newtheorem{remark}{Remark}[section]
\newtheorem{definition}{Definition}[section]
\newtheorem{lemma}{Lemma}[section]
\newtheorem{corollary}{Corollary}[section]
\newtheorem{proposition}{Proposition}[section]

\allowdisplaybreaks

\DeclareMathOperator{\loc}{loc}

\makeatletter
\@addtoreset{equation}{section}
\makeatother
\makeatletter
\@addtoreset{equation}{section}
\makeatother

\title{Global well-posedness of the three-dimensional non-isentropic compressible magnetohydrodynamic equations under a scaling-invariant smallness condition
\thanks{Xu's research was partially supported by Postgraduate Research an Innovation Project of Southwest University (No. SWUB25035). Zhong's research was partially supported by National Natural Science Foundation of China (No. 12371227) and Fundamental Research Funds for the Central Universities (No. SWU--KU24001).}
}

\author{Lin Xu,\ Xin Zhong{\thanks{E-mail addresses: mathxu@email.swu.edu.cn (L. Xu), xzhong1014@amss.ac.cn (X. Zhong).}}
\date{}\\
\footnotesize School of Mathematics and Statistics, Southwest University, Chongqing 400715, P. R. China}
\begin{document}
\maketitle
\begin{abstract}
We consider the Cauchy problem of the non-isentropic compressible magnetohydrodynamic equations in $\mathbb{R}^3$ with far-field vacuum. By deriving delicate energy estimates and exploiting the intrinsic structure of the system, we establish the global existence and uniqueness of strong solutions provided that the scaling-invariant quantity
\begin{align*}
(1+\bar{\rho}+\tfrac{1}{\bar{\rho}}) [\|\rho_{0}\|_{L^{3}}+ ( \bar{\rho}^{2}+\bar{\rho})( \| \sqrt{\rho_{0}}u_{0}\|_{L^{2}}^{2}+\| b_{0}\|_{L^{2}}^{2}) ] [\|\nabla u_{0}\|_{L^{2}}^{2}+(\bar{\rho}+1)\|\sqrt{\rho_{0}} \theta_{0}\|_{L^{2}}^{2}+\| \nabla b_{0}\|_{L^{2}}^{2}+\|  b_{0}\|_{L^{4}}^{4} ]
\end{align*}
is sufficiently small, where $\bar{\rho}$ denotes the essential supremum of the initial density. Our result may be regarded as an improved version compared with that of Liu and the second author (J. Differential Equations 336 (2022), pp. 456--478) in the sense that an artificial condition $3\mu>\lambda$ on the viscosity coefficients is removed. In particular, we provide a new scaling-invariant quantity regarding the initial data.
\end{abstract}




\section{Introduction}
The motion of three-dimensional compressible, viscous, and heat-conducting magnetohydrodynamic (MHD) flows is governed by (see, e.g., \cite[Chapter 3]{LT12})
\begin{align}\label{1}
\begin{cases}
\rho_{t}+\operatorname{div}(\rho u)=0, \\[1mm]
(\rho u)_{t}+\operatorname{div}(\rho u \otimes u)-\mu \Delta u-(\lambda+\mu) \nabla \operatorname{div} u+\nabla P=(\operatorname{curl} b) \times b, \\[1mm]
c_{v} \rho\left(\theta_{t}+u \cdot \nabla \theta\right)+P \operatorname{div} u-\kappa \Delta \theta
=2 \mu|\mathfrak{D}(u)|^{2}+\lambda(\operatorname{div} u)^{2}+\nu|\operatorname{curl} b|^{2}, \\[1mm]
b_{t}-\operatorname{curl}(u \times b)=\nu \Delta b, \\[1mm]
\operatorname{div} b=0,
\end{cases}
\end{align}
where the unknowns $\rho \geq 0$, $u \in \mathbb{R}^{3}$, $\theta \geq 0$, and $b \in \mathbb{R}^{3}$ denote the density, velocity, absolute temperature, and magnetic field, respectively. The pressure $P$ is given by the ideal gas law
\begin{align*}
P = R \rho \theta,
\end{align*}
with the gas constant $R>0$. The deformation tensor $\mathfrak{D}(u)$ is defined by
\begin{align*}
\mathfrak{D}(u)=\frac{1}{2}\big(\nabla u+(\nabla u)^{\top}\big),
\end{align*}
where $(\nabla u)^{\top}$ denotes the transpose of $\nabla u$. The positive constants $c_{v}$, $\kappa$, and $\nu$ stand for the specific heat at constant volume, the heat conductivity coefficient, and the magnetic diffusivity, respectively. The viscosity coefficients $\mu$ and $\lambda$ are assumed to satisfy the physical restrictions
\begin{align*}
\mu>0, \quad 2 \mu+3 \lambda \geq 0.
\end{align*}

We consider the Cauchy problem of system \eqref{1} in $\mathbb{R}^{3} \times (0,T)$ with the initial condition
\begin{align}\label{4}
(\rho, u, \theta, b)(x, 0)=(\rho_{0}, u_{0}, \theta_{0}, b_{0})(x), \quad x \in \mathbb{R}^{3},
\end{align}
and the far-field behavior
\begin{align}\label{5}
(\rho, u, \theta, b)(x, t) \rightarrow (0,0,0,0)
\quad \text{as } |x| \rightarrow \infty,\ t>0.
\end{align}

The non-isentropic compressible MHD equations describe the motion of electrically conducting fluids such as plasmas, liquid metals, and electrolytes. They consist of a coupled system of non-isentropic compressible Navier--Stokes equations of fluid dynamics and Maxwell equations of electromagnetism. The system is not merely a superposition of fluid and magnetic field equations, but an intrinsically coupled model in which fluid motion and electromagnetic effects interact in a highly nonlinear way. This coupling leads to a rich structure and makes the analytical study particularly challenging.

There is huge literature on the studies of the non-isentropic compressible MHD equations, such as the global existence of weak solutions \cite{DF06,HW08,LG14}, low Mach number limits \cite{JJL12,JJLX14}, blow-up criteria of solutions \cite{HL13,W21,Z19}, and so on. It is also of great interest to study the global well-posedness of strong solutions to the system \eqref{1}. By energy method, Pu and Guo \cite{PG13} showed the global existence of smooth solutions when the initial data are close to a constant equilibrium in $H^{3}(\mathbb{R}^3)$. Moreover, they obtained the convergence rates of the $L^p$-norm of such solutions to the constant state when the $L^q$-norm of the perturbation is bounded. With the help of the $L^p$-theory of elliptic equations, the global well-posedness of strong solutions with vacuum to the Cauchy problem \eqref{1}--\eqref{5} was established in \cite{LZ20} provided that $3\mu>\lambda$ and $\|\rho_{0}\|_{L^{\infty}}+\|b_{0}\|_{L^{3}}$ is sufficiently small. Such a result was later generalized in \cite{HJP22} by removing an artificial restriction $3\mu>\lambda$ and allowing large oscillation as the smallness condition is imposed on $\|\rho_{0}\|_{L^{1}}+\|b_{0}\|_{L^{2}}$.
More recently, applying initial layer analysis and bootstrap arguments,
Liu and the second author \cite{LZ25} established the global well-posedness and algebraic decay rates of strong solutions to \eqref{1}--\eqref{5} provided that the initial data are of small total energy. Furthermore, the uniform boundedness of the entropy for \eqref{1} was obtained in \cite{LZ23} provided that the initial density decays suitably slow at infinity with the aid of modified De Giorgi type iteration techniques and singularly weighted estimates. Apart from the Cauchy problem with vacuum \cite{LZ20,HJP22,LZ25},
several authors dealt with the initial-boundary value problem of \eqref{1}. Under the condition that the initial data are of small energy,
Chen--Peng--Wang \cite{CPW23} derived the global well-posedness and large-time behavior of strong solutions in 3D exterior domains with Navier-slip boundary conditions for the velocity and magnetic field.
For the regular initial data with small energy, the global existence of classical and weak solutions as well as the exponential decay rate in 3D bounded domains with Navier-slip boundary conditions was proven in \cite{CCW26}, while it was shown in \cite{XYZ25} the global well-posedness and the exponential decay estimates of strong solutions in 3D bounded domains subject to Dirichlet boundary conditions for the velocity and temperature.

It should be noted that all results in \cite{CCW26,CPW23,XYZ25,LZ20,HJP22,LZ25} need some smallness condition depending on the initial data. Recently, under the additional restriction $3\mu>\lambda$ on the viscosity coefficients, Liu and the second author \cite{LZ22} demonstrated the global existence of strong solutions to the Cauchy problem \eqref{1}--\eqref{5} provided that
\begin{align}\label{sca111}
	\hat{\rho} \big[ \|\rho_{0} \|_{L^{3}}+\hat{\rho}^{2} ( \|\sqrt{\rho_{0}} u_{0} \|_{L^{2}}^{2}+ \|b_{0} \|_{L^{2}}^{2} ) \big]
	\big[ \|\nabla u_{0} \|_{L^{2}}^{2}+\hat{\rho} ( \|\sqrt{\rho_{0}} E_{0} \|_{L^{2}}^{2}+ \|\nabla b_{0} \|_{L^{2}}^{2} ) \big]
\end{align}
is properly small, where $E_{0}=\frac{\left|u_{0}\right|^{2}}{2}+c_{v} \theta_{0}$ and $\hat{\rho}=\left\|\rho_{0}\right\|_{L^{\infty}}+1$.
In particularly, {\it the smallness condition is independent of any initial data and the quantity in \eqref{sca111} is invariant} under the parabolic-type scaling
\begin{align}\label{sca}
\rho_{\tau}(x, t)\triangleq\rho(\tau x, \tau^{2} t),\
u_{\tau}(x, t)\triangleq\tau\,u(\tau x, \tau^{2} t),\
\theta_{\tau}(x, t)\triangleq\tau^{2}\,\theta(\tau x, \tau^{2} t),\
b_{\tau}(x, t)\triangleq\tau\,b(\tau x, \tau^{2} t),
\end{align}
for any $\tau>0$. The additional condition $3\mu>\lambda$ is mainly used to get the weighted spatial $L^4$-estimate of the velocity, which plays a crucial role in their analysis. However, from the viewpoint of physics, such a restriction seems unnatural as several fluids have bulk viscosities which are hundreds or thousands of times larger than their shear viscosities (see, e.g., \cite{C11}). Thus, removing artificial assumptions on the viscosity coefficients is not only a mathematical issue, but also physically relevant. This gives us the motivation to improve the result obtained in \cite{LZ22}. More precisely, under some smallness condition on a scaling-invariant quantity, we shall establish global strong solutions to the Cauchy problem \eqref{1}--\eqref{5} without any additional constraint on the viscosity coefficients.

Before stating our main result precisely, we describe the notation throughout the paper. For an integer $k \geq 0$ and $1 \leq q \leq \infty$, the standard Sobolev spaces are defined as
\begin{gather*}
D^{k,q}(\mathbb{R}^{3}) = \{ u \in L^1_{\loc}(\mathbb{R}^{3}) \mid \|\nabla^k u\|_{L^q} < \infty \}, \ \|u\|_{D^{k,q}} = \|\nabla^k u\|_{L^q},\ \|  u \|_{L^{q}}=\|u \|_{L^{q}(\mathbb{R}^{3})}, \\
W^{k,q}(\mathbb{R}^{3}) = L^q(\mathbb{R}^{3}) \cap D^{k,q}(\mathbb{R}^{3}), \ D^k(\mathbb{R}^{3}) = D^{k,2}(\mathbb{R}^{3}), \ H^k(\mathbb{R}^{3}) = W^{k,2}(\mathbb{R}^{3}).
\end{gather*}

Next, we give the definition of the strong solution to \eqref{1}--\eqref{5}.
\begin{definition}
(Strong solution) For $T>0$, $(\rho\geq0, u, \theta\geq0, b)$ is called a strong solution to the problem \eqref{1}--\eqref{5} in $\mathbb{R}^{3} \times[0, T]$, if for some $q \in(3,6]$,
\begin{align*}
\begin{cases}
\rho \in C ([0, T] ; W^{1, q} \cap H^{1} ),~~ \rho_{t} \in C ([0, T] ; L^{2} \cap L^{q} ), \\
(u, \theta, b) \in C ([0, T] ; D^{2} \cap D_{0}^{1} ) \cap L^{2} ([0, T] ; D^{2, q} ),\\
(u_{t}, \theta_{t}, b_{t} ) \in L^{2} ([0, T] ; D_{0}^{1} ),~~ (\sqrt{\rho} u_{t}, \sqrt{\rho} \theta_{t}, b_{t} ) \in L^{\infty} ([0, T] ; L^{2} ),
\end{cases}
\end{align*}
and  $(\rho, u, \theta, b)$  satisfies system \eqref{1} a.e. in  $\mathbb{R}^{3} \times(0, T]$.
\end{definition}

Our main result can be stated as follows.
\begin{theorem}\label{thm1}
Let $q\in(3,6]$ and assume that the initial data $(\rho_{0}\ge 0,u_{0},\theta_{0}\ge 0,b_{0})$ satisfies
\begin{align}\label{xz}
\rho_{0} \in H^{1} \cap W^{1,q},~
(\sqrt{\rho_{0}} u_{0}, \sqrt{\rho_{0}} \theta_{0} ) \in L^{2},~
(u_{0}, \theta_{0}) \in D_{0}^{1} \cap D^{2},~
b_{0} \in H^{2},~ \operatorname{div} b_{0}=0,
\end{align}
and the compatibility condition
\begin{align}\label{com}
\begin{cases}
-\mu \Delta u_{0}-(\lambda+\mu) \nabla \operatorname{div} u_{0}
+\nabla\!\left(R \rho_{0} \theta_{0}\right)-\operatorname{curl} b_{0} \times b_{0}=\sqrt{\rho_{0}}\, g_{1}, \\[1mm]
-\kappa \Delta \theta_{0}
-2 \mu|\mathfrak{D}(u_{0})|^{2}
-\lambda(\operatorname{div} u_{0})^{2}
-\nu\left|\operatorname{curl} b_{0}\right|^{2}
=\sqrt{\rho_{0}}\, g_{2},
\end{cases}
\end{align}
for some $g_{1},g_{2} \in L^{2}(\mathbb{R}^{3})$. There exists a positive constant $\varepsilon_{0}$ depending only on the parameters $R, \mu, \lambda, \nu, \kappa$, and $c_{v}$ such that if
\begin{align}\label{sca1}
N_{0}\triangleq \big(1+\bar{\rho}+\tfrac{1}{\bar{\rho}}\big)\,S_{0}^{\prime}S_{0}^{\prime\prime} \leq \varepsilon_{0},
\end{align}
then the Cauchy problem \eqref{1}--\eqref{5} admits a unique global strong solution $(\rho,u,\theta,b)$. Here $\bar{\rho}\triangleq\|\rho_{0} \|_{L^{\infty}}$,
\begin{gather*}
S_{0}^{\prime}\triangleq \|\rho_{0}\|_{L^{3}}+( \bar{\rho}^{2}+\bar{\rho})
\big( \| \sqrt{\rho_{0}}u_{0}\|_{L^{2}}^{2}+\| b_{0}\|_{L^{2}}^{2}\big),\\
S_{0}^{\prime\prime}\triangleq \|\nabla u_{0}\|_{L^{2}}^{2}
+(\bar{\rho}+1)\|\sqrt{\rho_{0}} \theta_{0}\|_{L^{2}}^{2}
+\| \nabla b_{0}\|_{L^{2}}^{2}
+\| b_{0}\|_{L^{4}}^{4}.
\end{gather*}
\end{theorem}

\begin{remark}
It is straightforward to verify that the quantity $N_{0}$ in \eqref{sca1} is invariant under the scaling transformation \eqref{sca}. More precisely, if $N_{0} \leq \varepsilon_{0}$ holds for some initial data $(\rho_{0},u_{0},\theta_{0},b_{0})$, then, for any $\tau\neq 0$, the rescaled initial data $(\rho_{\tau,0},u_{\tau,0},\theta_{\tau,0},b_{\tau,0})$ defined by \eqref{sca} at $t=0$ also satisfies $N_{0} \leq \varepsilon_{0}$. In particular, our result can be regarded as an improved version compared with that of Liu and the second author \cite{LZ22} in the sense that an artificial condition $3\mu>\lambda$ on the viscosity coefficients is removed. Moreover, \eqref{sca1} and \eqref{sca111} reveal that the scaling-invariant quantity may be not unique even under the same scaling transformation \eqref{sca}.
\end{remark}

\begin{remark}
Let $\tau>0$ and consider a general rescaling unknowns of the form
\begin{align*}
\rho^{\tau}(x,t) \triangleq \tau^{a}\rho(\tau x,\tau^{\gamma} t), \
u^{\tau}(x,t) \triangleq \tau^{b}u(\tau x,\tau^{\gamma} t), \
\theta^{\tau}(x,t) \triangleq \tau^{c}\theta(\tau x,\tau^{\gamma} t), \
b^{\tau}(x,t) \triangleq \tau^{d}b(\tau x,\tau^{\gamma} t),
\end{align*}
where $a,b,c,d$, and $\gamma$ are to be determined. A straightforward computation yields that
\begin{align*}
\gamma = 2,\ a=0,\ b=1,\ c=2,\ d=1.
\end{align*}
That is, \eqref{sca} is a unique nontrivial scaling-invariant transformation for the system \eqref{1}. This is very different from the case of non-isentropic compressible Navier--Stokes equations (see \cite{L20,W25}). Thus, the magnetic field plays a significant role.
\end{remark}

\begin{remark}
It should be noted that Cheng and Dong \cite{C24} obtained the local well-posedness of strong solutions of \eqref{1} without using the compatibility condition \eqref{com} via time-weighted energy estimates. A natural question then appears: is it possible to remove or weaken \eqref{com} in Theorem \ref{thm1}? New ideas are needed to deal with this issue, which will be left for future investigation.	
\end{remark}

We mainly employ bootstrap arguments (see Lemma \ref{boot}) and a blow-up criterion established in \cite{HL13} to prove Theorem \ref{thm1}. Similarly to the previous works on the 3D Cauchy problem \cite{LZ20,HJP22,LZ25,LZ22}, the key issue is to derive the uniform upper bound of the density. However, it should be pointed out that the methods used in the previous works  \cite{LZ20,LZ25,HJP22} cannot be adopted here since their smallness conditions are imposed directly on the initial data and hence are not compatible with the scaling-invariant framework.
Moreover, as stated above, the analysis in \cite{LZ22} depends heavily on the condition $3\mu>\lambda$. Therefore, to obtain the desired \emph{a priori} estimates under the scaling-invariant assumption \eqref{sca1} (without restriction $3\mu>\lambda$), we need some new observations and ideas.

First, motivated by \cite{L20}, we attempt to obtain the useful dissipation estimates from the basic energy inequality in terms of $L_{t}^{\infty}L_{x}^{3}$-norm of the density and $L_{t}^{2}L_{x}^{2}$-norm of $\nabla\theta$ (see \eqref{x0}).
With the help of {\it a priori hypothesis} \eqref{p1} and the structure of system \eqref{1}, we derive an estimate on the $L_{t}^{\infty}L_{x}^{3}$-norm of density (see Lemma \ref{lem2}). The next step is to obtain an estimate on the $L_{t}^{2}L_{x}^{2}$-norm of the gradient of temperature. At this stage, our argument relies on the temperature equation instead of the total energy equation used in \cite{LZ22}. This requires us to control the quantities $\sqrt{\rho}\dot{u}$ and $\nabla^{2}b$ in $L_{t}^{2}L_{x}^{2}$ (see Lemma \ref{lem3}). Multiplying \eqref{1}$_{2}$ by $u_{t}$ and applying delicate energy estimates, the estimate on the $L_{t}^{2}L_{x}^{2}$-norm for $\sqrt{\rho}\dot{u}$ can be obtained. It should be emphasized that, due to the term $\operatorname{curl} b\times b$, the estimate for $\sqrt{\rho}\dot{u}$ generates the complicated quantity $K_{2}(t)$ (see \eqref{xxxx0}). To control all these magnetic effects, one needs appropriate estimates for the magnetic field, which are derived from \eqref{1}$_{4}$ (see Lemma \ref{lem5}).

Next, the key issue is to enclose the bootstrap argument. Based on the estimates derived in Lemmas \ref{lem1}--\ref{lem5}, one obtains the uniform bound of density (via characteristic arguments) and the desired control of the quantity $S_{t}$ provided that the initial scaling-invariant quantity $N_{0}$ is sufficiently small (see Corollaries \ref{co1} and \ref{co2}). In particular, in order to construct a complete scaling-invariant quantity that meets the requirements of the \emph{a priori} estimates and the bootstrap argument, we multiply \eqref{c9} by $(1+\bar{\rho}+\tfrac{1}{\bar{\rho}})$. Moreover, we succeed in establishing the estimates for $\|\rho\|_{L^{3}}$ in \eqref{c3} and $\|\nabla u \|_{L^{2}}^{2}$ in \eqref{c9}. During this process, because of the strong nonlinear coupling between the fluid velocity, temperature, and magnetic field, the key techniques in \cite{W25} cannot be applied in a straightforward manner. A new ingredient in our analysis is the introduction of the quantity
$$
(\bar{\rho}^{2}+\bar{\rho})\sup_{t\in[0,T]}\|\rho(t)\|_{L^{3}}\int_{0}^{T}\|\nabla\theta\|_{L^{2}}^{2} dt+(\bar{\rho}+1)\sup _{t \in[0, T]} \|\nabla b(t)\|_{L^{2}}^{2} \int_{0}^{T}\|\nabla u\|_{L^{2}}^{2} d t
+\sup _{t \in[0, T]}\|\nabla u(t)\|_{L^{2}}^{2}\int_{0}^{T}\|\nabla u\|_{L^{2}}^{2}\, d t,
$$
which is assumed to be sufficiently small (see \eqref{p1}). This structure allows us to completes the bootstrap argument when the condition \eqref{sca1} holds.

The rest of this paper is organized as follows. Some important inequalities and auxiliary lemmas will be given in Section \ref{sec2}. Section \ref{sec3} is devoted to proving Theorem \ref{thm1}.

\section{Preliminaries}\label{sec2}
In this section, some known facts and auxiliary inequalities used frequently later are collected. We begin with the local existence of a unique strong solution to the problem \eqref{1}--\eqref{5} (see \cite{FY09}).
\begin{lemma}\label{local}
Assume that the initial data $(\rho_{0}, u_{0}, \theta_{0}, b_{0})$ satisfies \eqref{xz} and \eqref{com}. Then there exists a positive time $T_{0}>0$ such that the problem \eqref{1}--\eqref{5} admits a unique strong solution in $\mathbb{R}^{3} \times\left(0, T_{0}\right]$.
\end{lemma}

Next, we recall the following abstract bootstrap principle (see \cite[Proposition 1.21]{T06}).
\begin{lemma}\label{boot}
Let $I$ be a time interval. For each $t \in I$, suppose that we have two statements, a  hypothesis $\mathbf{H}(t)$ and a conclusion $\mathbf{C}(t)$. Suppose we can verify the following four assertions:
\begin{enumerate}
\item[$(a)$]\, If $\mathbf{H}(t)$ is true for some time $t \in I$, then $\mathbf{C}(t)$ is also true for that time $t$.
\item[$(b)$]\, If $\mathbf{C}(t)$ is true for some $t \in I$, then  $\mathbf{H} (t^{\prime} )$ is true for all $t^{\prime} \in I$ in a neighborhood of $t$.
	\item[$(c)$]\, If $t_{1}$, $t_{2}$, $\ldots$ is a sequence of times in $I$ which converges to another time $t \in I$, and $\mathbf{C}(t_{n})$ is true for all  $t_{n}$, then $\mathbf{C}(t)$ is true.
\item[$(d)$]\, $\mathbf{H}(t)$ is true for at least one time $t \in I$.
\end{enumerate}
Then $\mathbf{C}(t)$ is true for all $t \in I$.
\end{lemma}

Then, we need the well-known Gagliardo--Nirenberg inequality (see \cite[Theorem 12.87]{book17}).
\begin{lemma}\label{gn}
Let  $1 \leq p$, $q \leq \infty$, $m \in \mathbb{N}$, $k \in \mathbb{N}_{0}$, with  $0 \leq k<m$, and let  $a$, $r$  be such that
\begin{align*}
0 \leq a \leq 1-\frac{k}{m}
\end{align*}
and
\begin{align*}
(1-a)\bigg(\frac{1}{p}-\frac{m-k}{3}\bigg)+a\bigg(\frac{1}{q}+\frac{k}{3}\bigg)=\frac{1}{r} \in(-\infty,1].
\end{align*}
Then there exists a constant  $C=C(m, p, q, a, k)>0$  such that
\begin{align}\label{gn1}
\|\nabla^{k} f \|_{L^{r}} \leq C\|f\|_{L^{q}}^{a} \|\nabla^{m} f \|_{L^{p}}^{1-a}
\end{align}
for every  $f \in L^{q} (\mathbb{R}^{3} ) \cap D^{m, p} (\mathbb{R}^{3} )$, with the following exceptional cases:

$(i)$ If $ k=0$, $m p<3$, and  $q=\infty$, we assume that  $f$  vanishes at infinity.

$(ii)$ If  $1<p<\infty$  and  $m-k-\frac{3}{p}$  is a non-negative integer, then \eqref{gn1} only holds for  $0<a \leq 1-\frac{k}{m}$.
\end{lemma}

Some special cases of \eqref{gn1} used frequently are
\begin{gather*}
\|f\|_{L^{6}} \leq C \|\nabla f\|_{L^{2}}, \quad
\|f\|_{L^{\infty}} \leq C\|f\|_{L^{6}}^{\frac{1}{2}} \|\nabla f \|_{L^{6}}^{\frac{1}{2}}
\leq C\|\nabla f\|_{L^{2}}^{\frac{1}{2}} \|\nabla^{2} f \|_{L^{2}}^{\frac{1}{2}},\\
\|\nabla f\|_{L^{2}} \leq C\|f\|_{L^{2}}^{\frac{1}{2}} \|\nabla^{2} f \|_{L^{2}}^{\frac{1}{2}},\quad
\|f\|_{L^{\infty}} \leq C\|f\|_{L^{2}}^{\frac{1}{4}} \|\nabla^{2} f \|_{L^{2}}^{\frac{3}{4}},\quad
\|f\|_{L^{3}} \leq C\|f\|_{L^{2}}^{\frac{1}{2}}\|\nabla f\|_{L^{2}}^{\frac{1}{2}}.
\end{gather*}

The effective viscous flux, the material derivative of the velocity, and the vorticity are defined respectively by
\begin{align*}
F \triangleq (2 \mu+\lambda)\operatorname{div} u - P - \frac{1}{2}|b|^{2},\quad
\dot{u} \triangleq u_{t} + u \cdot \nabla u,\quad
\operatorname{curl} u \triangleq \nabla \times u.
\end{align*}
Since $\operatorname{div} b = 0$ and $b\cdot \nabla b = \operatorname{div}(b\otimes b)$, applying $\operatorname{div}$ and $\operatorname{curl}$ to \eqref{1}$_{2}$ yields the elliptic equations
\begin{align}\label{equ}
\Delta F = \operatorname{div}\big(\rho \dot{u} - b \cdot \nabla b\big), \quad
\mu \Delta (\operatorname{curl} u) = \operatorname{curl}\big(\rho \dot{u} - b \cdot \nabla b\big).
\end{align}

Finally, the next lemma gives basic estimates for $F$ and $u$.
\begin{lemma}\label{lem}
Let $(\rho, u, \theta, b)$ be a smooth solution of the problem \eqref{1}--\eqref{5}.  Then there exists a positive constant $C$ depending only on $\mu$ and $\lambda$ such that
\begin{gather}
\|\nabla F\|_{L^{2}}+\|\nabla \operatorname{curl} u\|_{L^{2}} \leq C\Big(\|\rho\|_{L^\infty}^{\frac12}\|\sqrt{\rho} \dot{u}\|_{L^{2}}+\||b||\nabla b|\|_{L^{2}}\Big), \label{eq1}\\
\|F\|_{L^{3}}+\|\operatorname{curl} u\|_{L^{3}}  \leq C\Big(\| \rho \|_{L^{3}}^{\frac{1}{2}}\|\sqrt{\rho} \dot{u}\|_{L^{2}}+ \|\nabla b\|_{L^{2}}^{2}\Big), \label{eq2}\\
\|\nabla u\|_{L^{6}} \leq C\Big(\|\rho\|_{L^\infty}^{\frac12}\|\sqrt{\rho} \dot{u}\|_{L^{2}}+\|\rho\|_{L^\infty}\|\nabla \theta\|_{L^{2}}+\| |b||\nabla b| \|_{L^{2}}\Big).\label{eq3}
\end{gather}
\end{lemma}
\begin{proof}
The standard $L^{2}$ theory for the elliptic equations \eqref{equ} directly implies \eqref{eq1}. Moreover,
\begin{align*}
\|\nabla u\|_{L^{6}} & \leq C\big(\|\operatorname{curl} u\|_{L^{6}}+\|\operatorname{div} u\|_{L^{6}}\big) \\
& \leq C\big(\|\operatorname{curl} u\|_{L^{6}}+\|F\|_{L^{6}}+\|\rho \theta\|_{L^{6}}+\||b|^{2}\|_{L^{6}}\big) \\
& \leq C\big(\|\nabla \operatorname{curl} u\|_{L^{2}}+\|\nabla F\|_{L^{2}}+\|\rho\|_{L^\infty}\|\nabla \theta\|_{L^{2}}+\||b||\nabla b|\|_{L^{2}}\big)\\
&\leq C\Big(\|\rho\|_{L^\infty}^{\frac12}\|\sqrt{\rho} \dot{u}\|_{L^{2}}+\|\rho\|_{L^\infty}\|\nabla \theta\|_{L^{2}}+\||b||\nabla b|\|_{L^{2}}\Big)
\end{align*}
 gives \eqref{eq3}. The representations of $F$ and $\operatorname{curl} u$ together with Sobolev's inequality yield that
\begin{align*}
\|F\|_{L^{3}}+\|\operatorname{curl} u\|_{L^{3}} \leq & C\| (-\Delta)^{-1} \operatorname{div}(\rho \dot{u}-b\cdot \nabla b)\|_{L^{3}}+C\| (-\Delta)^{-1} \operatorname{curl}  (\rho \dot{u}-b\cdot \nabla b)\|_{L^{3}} \\
\leq&  C\|\rho \dot{u}+|b||\nabla b|\|_{L^{\frac{3}{2}}} \\
\leq &  C\| \rho \|_{L^{3}}^{\frac{1}{2}}\|\sqrt{\rho} \dot{u}\|_{L^{2}}+C\|\nabla b\|_{L^{2}}^{2},
\end{align*}
as the desired \eqref{eq2}.
\end{proof}

\section{Proof of Theorem \ref{thm1}}\label{sec3}
In this section, using the local well-posedness result established in Lemma \ref{local}, we prove the global existence and uniqueness of strong solutions to problem \eqref{1}--\eqref{5}. Throughout this section, $(\rho,u,\theta,b)$ denotes a local strong solution to \eqref{1}--\eqref{5} in $\mathbb{R}^{3}\times(0,T]$ for some $T>0$. For convenience, we write
\begin{align}\label{xz2}
\int \cdot \, d x := \int_{\mathbb{R}^{3}} \cdot \, d x, \quad \nu = \kappa = c_{v} = R = 1.
\end{align}

For later use, we introduce
\begin{align*}
	S_{t} \triangleq \bigg(1+\bar{\rho}+\frac{1}{\bar{\rho}}\bigg) S_{t}^{\prime} S_{t}^{\prime\prime},
\end{align*}
where
\begin{align*}
S_{t}^{\prime} \triangleq \|\rho(t)\|_{L^{3}} + (\bar{\rho}^{2}+\bar{\rho})\big( \| \sqrt{\rho} u(t)\|_{L^{2}}^{2} + \| b(t)\|_{L^{2}}^{2}\big),
\end{align*}
and
\begin{align*}
S_{t}^{\prime\prime} \triangleq \|\nabla u(t)\|_{L^{2}}^{2} + (\bar{\rho}+1)\|\sqrt{\rho} \theta(t)\|_{L^{2}}^{2}
+ \|\nabla b(t)\|_{L^{2}}^{2} + \|b(t)\|_{L^{4}}^{4}.
\end{align*}

The next proposition is the key step in the proof of Theorem \ref{thm1}.
\begin{proposition}\label{pro}
Under the hypotheses of Theorem \ref{thm1}, $\epsilon_{0}>0$ is the small positive constant that appears in \eqref{sca1}. Then there exist positive constants $C_{1}$ and $\epsilon_{2}$, depending only on the physical parameters (but independent of the initial data, $\bar{\rho}$, and $T$), with $\epsilon_{2}$ further being determined by \eqref{sm1}, \eqref{sm2},  \eqref{sm4},  \eqref{sm6}, \eqref{sm11}, \eqref{sm7}, and \eqref{sm9}, such that, if for all $t\in[0,T]$,
\begin{align}\label{p1}
	\begin{cases}
		\rho \leq 2\bar{\rho}, \\[2mm]
		S_{t} + (\bar{\rho}^{2}+\bar{\rho})
		\sup_{t\in[0,T]} \|\rho(t)\|_{L^{3}}\int_{0}^{T} \|\nabla \theta\|_{L^{2}}^{2}\, dt+(\bar{\rho}+1)\sup _{t \in[0, T]} \|\nabla b(t)\|_{L^{2}}^{2} \int_{0}^{T}\|\nabla u\|_{L^{2}}^{2}\, d t\\[2mm]
		+\sup _{t \in[0, T]}\|\nabla u(t)\|_{L^{2}}^{2}\int_{0}^{T}\|\nabla u\|_{L^{2}}^{2}\, d t
		\leq 2\epsilon,
	\end{cases}
\end{align}
for some $\epsilon$ satisfying
\begin{align*}
C_{1}\epsilon_{0} \leq \epsilon \leq \epsilon_{2},
\end{align*}
then
\begin{align}\label{p2}
	\begin{cases}
		\rho \leq \frac{3}{2}\bar{\rho}, \\[2mm]
		S_{t} + (\bar{\rho}^{2}+\bar{\rho})
		\sup_{t\in[0,T]} \|\rho(t)\|_{L^{3}}\int_{0}^{T} \|\nabla \theta\|_{L^{2}}^{2}\, dt+(\bar{\rho}+1)\sup _{t \in[0, T]} \|\nabla b(t)\|_{L^{2}}^{2} \int_{0}^{T}\|\nabla u\|_{L^{2}}^{2}\, d t\\[2mm]
		+\sup _{t \in[0, T]}\|\nabla u(t)\|_{L^{2}}^{2}\int_{0}^{T}\|\nabla u\|_{L^{2}}^{2}\, d t
		\leq \frac{3}{2}\epsilon.
	\end{cases}
\end{align}
\end{proposition}

The proof of Proposition \ref{pro} is based on Corollaries \ref{co1} and \ref{co2}. Throughout the rest of the paper, the symbol $C$ denotes a generic positive constant independent of the initial data, $T$, $\epsilon_{0}$, $\epsilon$, and $\bar{\rho}$, whose value may change from line to line. As a first step, we derive a basic energy estimate.
\begin{lemma}\label{lem1}
It holds that, for all $t \in[0, T]$,
\begin{align}\label{x0}
& \|\sqrt{\rho} u(t)\|_{L^{2}}^{2}+\|b(t)\|_{L^{2}}^{2}
+\int_{0}^{T}\big( \|\nabla u \|_{L^{2}}^{2}+\|\nabla b \|_{L^{2}}^{2}\big)\, d t \notag\\
&\leq C\big(\|\sqrt{\rho_{0}} u_{0} \|_{L^{2}}^{2}+ \|b_{0} \|_{L^{2}}^{2}\big)
+C\sup _{t \in[0, T]}\|\rho(t)\|_{L^{3}}^{2} \int_{0}^{T}\|\nabla \theta \|_{L^{2}}^{2}\, d t.
\end{align}
\end{lemma}

\begin{proof}
Multiplying \eqref{1}$_2$ by $u$ and \eqref{1}$_4$ by $b$, respectively, adding the resulting equations together, and integration by parts, we obtain from H\"{o}lder's and Sobolev's inequalities that
\begin{align*}
&\frac{1}{2} \frac{d}{d t}\big(\|\sqrt{\rho} u\|_{L^{2}}^{2}+\|b\|_{L^{2}}^{2}\big)
+\mu\|\nabla u\|_{L^{2}}^{2}+(\mu+\lambda)\|\operatorname{div} u\|_{L^{2}}^{2}
+\|\nabla b\|_{L^{2}}^{2} \\
& =\int \rho \theta \operatorname{div} u \, d x
\leq \|\rho\|_{L^{3}}\|\theta\|_{L^{6}}\|\operatorname{div} u\|_{L^{2}}
\leq (\mu+\lambda)\|\operatorname{div} u\|_{L^{2}}^{2}
+C\|\rho\|_{L^{3}}^{2}\|\nabla \theta\|_{L^{2}}^{2},
\end{align*}
which yields that
\begin{align}\label{x1}
\frac{d}{d t}\left(\|\sqrt{\rho} u\|_{L^{2}}^{2}+\|b\|_{L^{2}}^{2}\right)
+\mu\|\nabla u\|_{L^{2}}^{2}+\|\nabla b\|_{L^{2}}^{2}
\leq C\|\rho\|_{L^{3}}^{2}\|\nabla \theta\|_{L^{2}}^{2}.
\end{align}
Integrating \eqref{x1} over $(0,T)$ leads to \eqref{x0}.
\end{proof}

Next, we show the following key estimate of $\|\rho \|_{L^{3}}$.
\begin{lemma}\label{lem2}
Under the hypotheses of Proposition \ref{pro},   it holds that, for all $t \in[0, T]$,
\begin{align}\label{xx1}
\|\rho(t)\|_{L^{3}} \leq C \|\rho_{0} \|_{L^{3}}+C(\bar{\rho}^{2}+\bar{\rho}) \int_{0}^{T}\big(\|\nabla u\|_{L^{2}}^{2}+\|\nabla b\|_{L^{2}}^{2}\big) \, d t.
\end{align}
\end{lemma}

\begin{proof}
Applying the operator  $(-\Delta)^{-1} \operatorname{div}$  to  \eqref{1}$_{2}$ gives that
\begin{align}\label{xx2}
 (-\Delta)^{-1} \operatorname{div}(\rho u)_{t} +(-\Delta)^{-1} \operatorname{div} \operatorname{div}(\rho u \otimes u)+(2 \mu+\lambda) \operatorname{div} u-P-\frac{|b|^{2}}{2}= (-\Delta)^{-1} \operatorname{div} \operatorname{div}(b \otimes b).
\end{align}
It follows from \eqref{1}$_{1}$ that
\begin{align}\label{xx3}
\partial_{t} \rho^{3}+\operatorname{div} (u \rho^{3} )+2 \operatorname{div} u \rho^{3}=0.
\end{align}
Then, multiplying \eqref{xx2} by $\rho^{3}$ and using \eqref{xx3}, we obtain that
\begin{align}\label{xx4}
	&\rho^{3}(-\Delta)^{-1} \operatorname{div}(\rho u)_{t}+\rho^{3} (-\Delta)^{-1} \operatorname{div} \operatorname{div}(\rho u \otimes u)-\frac{2 \mu+\lambda}{2}\big(\partial_{t} \rho^{3}+\operatorname{div} (u \rho^{3} )\big)-\rho^{3} \bigg(P+\frac{|b|^{2}}{2}\bigg) \notag\\
	&=\rho^{3}(-\Delta)^{-1} \operatorname{div} \operatorname{div}(b \otimes b).
\end{align}
By \eqref{xx3} and integration by parts, one sees that
\begin{align*}
	&\int \rho^{3}  (-\Delta)^{-1} \operatorname{div}(\rho u)_{t}\, d x\notag\\
	&=  \frac{d}{d t} \int \rho^{3}  (-\Delta)^{-1} \operatorname{div}(\rho u) d x+\int\big(\operatorname{div} (\rho^{3} u )+2 \operatorname{div} u \rho^{3}\big)  (-\Delta)^{-1} \operatorname{div}(\rho u) \, d x \notag\\
	&=   \frac{d}{d t} \int \rho^{3}  (-\Delta)^{-1} \operatorname{div}(\rho u) \, d x+\int\big[2 \operatorname{div} u \rho^{3}  (-\Delta)^{-1} \operatorname{div}(\rho u)-\rho^{3} u \cdot \nabla (-\Delta)^{-1} \operatorname{div}(\rho u)\big] d x,
\end{align*}
which combined with \eqref{xx4} implies that
\begin{align}\label{xx5}
\frac{d}{d t} \int\bigg(\frac{2 \mu+\lambda}{2}-(-\Delta)^{-1} \operatorname{div}(\rho u)\bigg) \rho^{3}\, d x+\int \rho^{3} \bigg(P+\frac{|b|^{2}}{2}\bigg) \, d x \triangleq J_{1}+J_{2},
\end{align}	
where
\begin{align*}
J_{1}=\int\big[  \rho^{3} (-\Delta)^{-1} \operatorname{div} \operatorname{div}(\rho u \otimes u)-\rho^{3}u \cdot \nabla  (-\Delta)^{-1} \operatorname{div}(\rho u)+2 \operatorname{div} u \rho^{3}  (-\Delta)^{-1}\operatorname{div}(\rho u)\big]  \, d x,
\end{align*}
and
\begin{align*}
J_{2}= -\int \rho^{3} (-\Delta)^{-1} \operatorname{div} \operatorname{div}(b \otimes b) \, d x.
\end{align*}

By H\"{o}lder's inequality, Sobolev's inequality, and \eqref{p1}, we have
\begin{align*}
	&\left|J_{1}\right|\leq C\|\rho\|_{L^{\infty}}^{2}\|\rho\|_{L^{3}}^{2}\|\nabla u\|_{L^{2}}^{2}+C\|\rho\|_{L^{\infty}}^{2}\|\rho\|_{L^{3}}\|\nabla u\|_{L^{2}}\|\rho u\|_{L^{2}}  \leq C\bar{\rho}^{2}\|\rho\|_{L^{3}}^{2}\|\nabla u\|_{L^{2}}^{2}, \\
	&\left|J_{2}\right| \leq C\|\rho\|_{L^{\infty}}\|\rho\|_{L^{3}}^{2}\| (-\Delta)^{-1}\operatorname{div} \operatorname{div}(b \otimes b)\|_{L^{3}} \leq C\bar{\rho}\|\rho\|_{L^{3}}^{2}\|\nabla b\|_{L^{2}}^{2}.
\end{align*}
Therefore, one gets that
\begin{align}\label{xx6}
&\frac{d}{d t} \int\bigg(\frac{2 \mu+\lambda}{2} -(-\Delta)^{-1} \operatorname{div}(\rho u)\bigg) \rho^{3}\, d x+\int \rho^{3} \bigg(P+\frac{|b|^{2}}{2}\bigg)\, d x\notag\\
 &\leq C\bar{\rho}^{2}\|\rho\|_{L^{3}}^{2}\|\nabla u\|_{L^{2}}^{2}+C\bar{\rho}\|\rho\|_{L^{3}}^{2}\|\nabla b\|_{L^{2}}^{2}\notag\\
&\leq C(\bar{\rho}^{2}+\bar{\rho}) \|\rho\|_{L^{3}}^{2}\big(\|\nabla u\|_{L^{2}}^{2}+\|\nabla b\|_{L^{2}}^{2}\big).
\end{align}
Integrating \eqref{xx6} in $t$ over $(0, T)$ leads to
\begin{align}\label{xx7}
\|\rho(t)\|_{L^{3}}^{3}+\int_{0}^{T}\int\rho^{3}\bigg(P+\frac{|b|^{2}}{2}\bigg)\, dxdt
\leq & C\|\rho_{0}\|_{L^{3}}^{3}+C \bar{\rho}^{\frac{3}{4}} \sup _{t \in[0, T]}\Big(\|\sqrt{\rho} u(t)\|_{L^{2}}^{\frac{1}{2}}\|\nabla u(t)\|_{L^{2}}^{\frac{1}{2}}\Big) \|\rho(t) \|_{L^{3}}^{3}\notag\\
&+C(\bar{\rho}^{2}+\bar{\rho}) \int_{0}^{T}\|\rho\|_{L^{3}}^{2}\big(\|\nabla u\|_{L^{2}}^{2}+\|\nabla b\|_{L^{2}}^{2}\big)\, d t,
\end{align}
due to
\begin{align}\label{xu0}
	\|(-\Delta)^{-1} \operatorname{div}(\rho u)\|_{L^{\infty}} & \leq C\|(-\Delta)^{-1} \operatorname{div}(\rho u)\|_{L^{6}}^{\frac{1}{2}}\|\nabla(-\Delta)^{-1} \operatorname{div}(\rho u)\|_{L^{6}}^{\frac{1}{2}} \notag\\
	& \leq C\|\rho u\|_{L^{2}}^{\frac{1}{2}}\|\rho u\|_{L^{6}}^{\frac{1}{2}}\leq C \bar{\rho}^{\frac{3}{4}}\|\sqrt{\rho} u\|_{L^{2}}^{\frac{1}{2}}\|\nabla u\|_{L^{2}}^{\frac{1}{2}}.
\end{align}

Choosing $\epsilon>0$ in \eqref{p1} sufficiently small such that
\begin{align}\label{sm1}
C \bar{\rho}^{\frac{3}{4}}\sup _{t \in[0, T]}\Big(\|\sqrt{\rho} u(t)\|_{L^{2}}^{\frac{1}{2}}\|\nabla u(t)\|_{L^{2}}^{\frac{1}{2}}\Big)
\leq \frac{1}{2},
\end{align}
then it follows from \eqref{xx7} that
\begin{align*}
\|\rho(t)\|_{L^{3}}^{3} \leq C \|\rho_{0}\|_{L^{3}}^{3}+C(\bar{\rho}^{2}+\bar{\rho}) \sup _{t \in[0, T]}\|\rho(t)\|_{L^{3}}^{2} \int_{0}^{T}\big(\|\nabla u\|_{L^{2}}^{2}+\|\nabla b\|_{L^{2}}^{2}\big)\, d t,
\end{align*}
which together with Young's inequality implies the desired \eqref{xx1}.
\end{proof}

Next, we derive a crucial estimate for $\int_{0}^{T}\|\nabla \theta\|_{L^{2}}^{2} \, dt$.
\begin{lemma}\label{lem3}
Under the hypotheses of Proposition \ref{pro}, it holds that, for all $t \in[0, T]$,
\begin{align}\label{xxx1}
	 &\|\sqrt\rho \theta(t)\|_{L^{2}}^{2}+\int_{0}^{T}\|\nabla \theta\|_{L^{2}}^{2} \, d t\notag\\
	 &\leq C \|\sqrt{\rho_{0}} \theta_{0}\|_{L^{2}}^{2}+C \sup _{t \in[0, T]}\big(\|\nabla u(t)\|_{L^{2}}^{2}\|\rho(t)\|_{L^{3}}\big) \int_{0}^{T}\|\sqrt{\rho} \dot{u}\|_{L^{2}}^{2}\, d t\notag\\
	&\quad +C \sup _{t \in[0, T]}\big(\|b(t)\|_{L^{2}} \|\nabla b(t)\|_{L^{2}}\big)\int_{0}^{T}\|\nabla^{2} b\|_{L^{2}}^{2} \, d t  +C  \int_{0}^{T}\|\nabla u\|_{L^{2}}^{2}\|\nabla b\|_{L^{2}}^{4}\, d t.
\end{align}
\end{lemma}
\begin{proof}
Multiplying \eqref{1}$_{3}$  by  $\theta$ and integrating the resultant over  $\mathbb{R}^{3}$ gives that
\begin{align}\label{xxx2}
\frac{1}{2} \frac{d}{d t} \|\sqrt\rho \theta\|_{L^{2}}^{2}+ \|\nabla \theta\|_{L^{2}}^{2}
	=-\int \rho \theta^{2} \operatorname{div} u \, dx+\int \mathcal{Q}(\nabla u) \theta \, dx+ \int|\operatorname{curl} b|^{2} \theta \, dx \triangleq\sum_{i=1}^{3} I_{i},
\end{align}
where  $\mathcal{Q}(\nabla u)\triangleq2 \mu|\mathfrak{D}(u)|^{2}+\lambda(\operatorname{div} u)^{2}$.
It follows from H\"{o}lder's inequality and Sobolev's inequality that
\begin{align}\label{xxx3}
I_{1} \leq C\|\rho\|_{L^{6}}\|\nabla u\|_{L^{2}}\|\theta^{2}\|_{L^{3}} \leq C\bar{\rho}^{\frac{1}{2}}\|\rho\|_{L^{3}}^{\frac{1}{2}}\|\nabla u\|_{L^{2}}\|\nabla \theta\|_{L^{2}}^{2}.
\end{align}
By virtue of \eqref{eq2} and H\"{o}lder's inequality, one has that
\begin{align}\label{xxx4}
	I_{2} \leq &  C\|\nabla \theta\|_{L^{2}}\|\nabla u\|_{L^{\frac{12}{5}}}^{2} \notag\\
	\leq &  C\|\nabla \theta\|_{L^{2}}\Big(\|\operatorname{div} u\|_{L^{\frac{12}{5}}}^{2}+ \|\operatorname{curl} u\|_{L^{\frac{12}{5}}}^{2}\Big)\notag\\
	\leq &  C\|\nabla \theta\|_{L^{2}}\big(\|\operatorname{div} u\|_{L^{2}}\|\operatorname{div} u\|_{L^{3}}+ \|\operatorname{curl} u\|_{L^{2}}\|\operatorname{curl} u\|_{L^{3}}\big)\notag\\
	\leq &  C\|\nabla \theta\|_{L^{2}}
	\|\nabla u\|_{L^{2}}\big( \|F\|_{L^{3}}+\|\rho \theta\|_{L^{3}}+\|\nabla b\|_{L^{2}}^{2}+\|\operatorname{curl} u\|_{L^{3}}\big) \notag\\	
	\leq & C\|\rho\|_{L^{3}}^{\frac{1}{2}}\|\nabla \theta\|_{L^{2}}\|\nabla u\|_{L^{2}}\|\sqrt{\rho} \dot{u}\|_{L^{2}}+C \bar{\rho}^{\frac{1}{2}}\|\rho\|_{L^{3}}^{\frac{1}{2}}\|\nabla u\|_{L^{2}}\|\nabla \theta\|_{L^{2}}^{2} + C\|\nabla \theta\|_{L^{2}}\|\nabla u\|_{L^{2}}\|\nabla b\|_{L^{2}}^{2}.
\end{align}
Integration by parts together with H{\"o}lder's and Sobolev's inequalities leads to
\begin{align}\label{xxx5}
I_{3} &\leq C \int\big(|b||\nabla^{2} b||\theta|+|b||\nabla b||\nabla \theta|\big) \, d x \notag\\
&\leq C\|b\|_{L^{3}}\|\nabla^{2} b\|_{L^{2}}\|\nabla \theta\|_{L^{2}}\notag\\
&\leq C\|b\|_{L^{2}}^{\frac{1}{2}} \|\nabla b\|_{L^{2}}^{\frac{1}{2}}\|\nabla^{2} b\|_{L^{2}}\|\nabla \theta\|_{L^{2}}.
\end{align}

Thus, inserting \eqref{xxx3}--\eqref{xxx5} into \eqref{xxx2}, one gets that
\begin{align}\label{xxx6}
	 \frac{1}{2} \frac{d}{d t} \|\sqrt\rho \theta\|_{L^{2}}^{2}+ \|\nabla \theta\|_{L^{2}}^{2} &\leq C\|\rho\|_{L^{3}}^{\frac{1}{2}}\|\nabla \theta\|_{L^{2}}\|\nabla u\|_{L^{2}}\|\sqrt{\rho} \dot{u}\|_{L^{2}}+C \bar{\rho}^{\frac{1}{2}}\|\rho\|_{L^{3}}^{\frac{1}{2}}\|\nabla u\|_{L^{2}}\|\nabla \theta\|_{L^{2}}^{2} \notag\\
	&\quad +C\|b\|_{L^{2}}^{\frac{1}{2}}\|\nabla b\|_{L^{2}}^{\frac{1}{2}}\|\nabla^{2} b\|_{L^{2}}\|\nabla \theta\|_{L^{2}}+C\|\nabla \theta\|_{L^{2}}\|\nabla u\|_{L^{2}}\|\nabla^{2} b\|_{L^{2}}^{2}.
\end{align}
Choosing $\epsilon$ in \eqref{p1} sufficiently small such that
\begin{align}\label{sm2}
C \bar{\rho}^{\frac{1}{2}}\|\rho(t)\|_{L^{3}}^{\frac{1}{2}}\|\nabla u(t)\|_{L^{2}} \leq \frac{1}{2},
\end{align}
then one deduces from \eqref{xxx6} that
\begin{align}\label{xxx7}
\frac{d}{d t} \|\sqrt\rho \theta\|_{L^{2}}^{2}+ \|\nabla \theta\|_{L^{2}}^{2} \leq C\|\rho\|_{L^{3}}\|\nabla u\|_{L^{2}}^{2}\|\sqrt{\rho} \dot{u}\|_{L^{2}}^{2}+C\|b\|_{L^{2}} \|\nabla b\|_{L^{2}}\|\nabla^{2} b\|_{L^{2}}^{2}+C\|\nabla u\|_{L^{2}}^{2}\|\nabla b\|_{L^{2}}^{4}.
\end{align}
Integrating \eqref{xxx7} in $t$ over $(0, T)$, we obtain the desired \eqref{xxx1}.
\end{proof}

Then, we get the crucial estimate on $\int_{0}^{T} \|\sqrt{\rho} \dot{u}\|_{L^{2}}^{2} d t$.
\begin{lemma}\label{lem4}
Under the hypotheses of Proposition \ref{pro},  it holds that, for all $t \in[0, T]$,
\begin{align}\label{xxxx0}
	 \|\nabla u(t)\|_{L^{2}}^{2}+\int_{0}^{T} \|\sqrt{\rho} \dot{u}\|_{L^{2}}^{2} d t &\leq C \big(\|\nabla u_{0}\|_{L^{2}}^{2}+\bar{\rho} \|\sqrt{\rho_{0}}\theta_{0}\|_{2}^{2}+\|b_{0}\|_{L^{4}}^{4}\big)+C \bar{\rho} \|\sqrt{\rho}\theta(t)\|_{2}^{2}+C\|b(t)\|_{L^{4}}^{4}  \notag\\
	&\quad +C \sup _{t \in[0, T]} K_{1}(t)   \int_{0}^{T}\|\nabla \theta\|_{L^{2}}^{2} \, d t+C\int_{0}^{T} K_{2}(t) \, d t,
\end{align}
where \begin{align*}
K_{1}(t)\triangleq\bar{\rho}+\bar{\rho}^{2}\|\rho(t)\|_{L^{3}}\|\nabla u(t)\|_{L^{2}}^{2},
\end{align*}
and
\begin{align*}
K_{2}(t)&\triangleq C\Big(\bar{\rho}\|b\|_{L^{2}}\|\nabla b\|_{L^{2}}\|b_{t}\|_{L^{2}}^{2} + \|\rho\|_{L^{3}}\|\nabla u\|_{L^{2}}^{2}\||b|| \nabla b |\|_{L^{2}}^{2}+  \bar{\rho}\|b\|_{L^{2}}\|\nabla b\|_{L^{2}}\|\nabla^{2} b\|_{L^{2}}^{2} +  \bar{\rho}  \|\nabla u\|_{L^{2}}^{2}\|\nabla b\|_{L^{2}}^{4}  \\
&\quad +  \|\rho\|_{L^{3}}\|\nabla u\|_{L^{2}}^{2}\||b||\nabla b |\|_{L^{2}}^{2}+ \bar{\rho}\|b\|_{L^{2}}^{\frac{1}{2}}\|\nabla b\|_{L^{2}}^{\frac{1}{2}}\|b_{t}\|_{L^{2}}\|\nabla \theta\|_{L^{2}}+ \|b\|_{L^{2}}^{\frac{1}{2}}\|\nabla b\|_{L^{2}}^{\frac{1}{2}}\|b_{t}\|_{L^{2}}\||b||\nabla b|\|_{L^{2}}   \\
&\quad +  \|\nabla u\|_{L^{2}}\|\nabla b\|_{L^{2}}^{2}\||b|| \nabla b|  \|_{L^{2}}+ \bar{\rho}\|\nabla u\|_{L^{2}}\|\nabla \theta\|_{L^{2}}\|\nabla b\|_{L^{2}}^{2}+\| b\|_{L^{2}}^{\frac{1}{2}}\| \nabla b\|_{L^{2}}^{\frac{1}{2}}\| \nabla^{2}b\|_{L^{2}}\||b||\nabla b | \|_{L^{2}}   \\
&\quad +  \bar{\rho}^{\frac{1}{2}}\| \rho \|_{L^{3}}^{\frac{1}{2}} \|\nabla u\|_{L^{2}}\|\nabla \theta\|_{L^{2}}\||b||\nabla b | \|_{L^{2}}+\|\nabla \theta\|_{L^{2}}\| b \|_{L^{2}}^{\frac{1}{2}} \| \nabla b \|_{L^{2}}^{\frac{1}{2}} \| \nabla^{2} b \|_{L^{2}}+   \|b\|_{L^{2}} \|\nabla b\|_{L^{2}} \|\nabla^{2} b\|_{L^{2}}^{2}  \Big).
\end{align*}
\end{lemma}

\begin{proof}
Multiplying \eqref{1}$_{2}$ by $u_{t}$ and integrating the resultant over $\mathbb{R}^{3}$ gives that
\begin{align}\label{xxxx1}
	&\frac{1}{2} \frac{d}{d t} \big(\mu\|\nabla u\|_{L^{2}}^{2}+(\mu+\lambda)\|\operatorname{div} u\|_{L^{2}}^{2}\big)+\|\sqrt{\rho} \dot{u}\|_{L^{2}}^{2} \notag\\
	&= \frac{d}{dt} \int  P \operatorname{div} u \, dx -\int P_{t}  \operatorname{div} u \, dx+\int \rho(u \cdot \nabla) u \cdot \dot{u}\, dx+\int b\cdot \nabla b\cdot u_{t}\, dx -\frac{1}{2} \int \nabla|b|^{2} \cdot u_{t}\, d x.
\end{align}
Using the definition of $F$ and the identity
\begin{align*}
	\frac{d}{d t} \int P \operatorname{div} u \, d x-\int P_{t} \operatorname{div} u \, d x
		= & \frac{d}{d t} \int P \operatorname{div} u \, d x-\frac{1}{2(2 \mu+\lambda)} \frac{d}{d t}\|P\|_{L^{2}}^{2} \\
		& -\frac{1}{2 \mu+\lambda} \int P_{t} F \, d x-\frac{1}{2(2 \mu+\lambda)} \int P_{t}|b|^{2} d x \\
		= & \frac{1}{2(2 \mu+\lambda)} \frac{d}{d t}\|P\|_{L^{2}}^{2}+\frac{1}{2(2 \mu+\lambda)} \frac{d}{d t} \int P|b|^{2}\, d x\\
		&+\frac{1}{2 \mu+\lambda} \frac{d}{d t} \int P F \,  d x  -\frac{1}{2 \mu+\lambda} \int P_{t}\Big(F+\frac{1}{2}|b|^{2}\Big)\, d x,
\end{align*}
the right-hand side of \eqref{xxxx1} can be written as	
\begin{align*}
&\frac{1}{2(2 \mu+\lambda)} \frac{d}{d t}\|P\|_{L^{2}}^{2}+\frac{1}{2(2 \mu+\lambda)} \frac{d}{d t} \int P|b|^{2} \, d x+\frac{1}{2 \mu+\lambda} \frac{d}{d t} \int P F \, d x \notag\\
&-\frac{1}{2 \mu+\lambda} \int P_{t}\Big(F+\frac{1}{2}|b|^{2}\Big) \, d x+\int  \rho(u \cdot \nabla) u \cdot \dot{u} \, dx+\int b\cdot \nabla b\cdot u_{t} \, dx-\frac{1}{2} \int \nabla|b|^{2} \cdot u_{t}\,  d x.
\end{align*}

From  \eqref{1}$_{2}$ and \eqref{1}$_{3}$, the pressure $P$ (recall \eqref{xz2}) satisfies
\begin{align*}
	P_{t}=-\operatorname{div}(Pu)- P \operatorname{div} u+\Delta \theta+\mathcal{Q}(\nabla u)+|\operatorname{curl} b|^{2},
\end{align*}
which leads to
\begin{align}\label{xxxx2}
	\int P_{t}\Big(F+\frac{1}{2}|b|^{2}\Big) \, d x
		&=\int\Big[\big(\mathcal{Q}(\nabla u)-P \operatorname{div} u+|\operatorname{curl} b|^{2}\big) F+( P u- \nabla \theta) \cdot \nabla F\Big] \, d x \notag\\
		&\quad+\frac{1}{2} \int\Big[\big(\mathcal{Q}(\nabla u)-P \operatorname{div} u+|\operatorname{curl} b|^{2}\big)|b|^{2}+( P u- \nabla \theta) \cdot \nabla|b|^{2}\Big] \, d x.
\end{align}
Using  $\|\nabla u\|_{L^{2}}^{2}=\|\operatorname{curl} u\|_{L^{2}}^{2}+\|\operatorname{div} u\|_{L^{2}}^{2}$ and $\int \rho(u \cdot \nabla) u \cdot \dot{u}\, dx\leq \frac{1}{2}\|\sqrt{\rho} \dot{u}\|_{L^{2}}^{2}+C\int  \rho|u|^{2}|\nabla u|^{2}\,  dx$, one obtains from \eqref{xxxx1} and \eqref{xxxx2} that
\begin{align}\label{xxxx3}
&\frac{1}{2} \frac{d}{d t}\bigg(\mu\|\operatorname{curl} u\|_{L^{2}}^{2}+\frac{\|F\|_{L^{2}}^{2}}{2 \mu+\lambda}+\frac{1}{(2 \mu+\lambda)} \int|b|^{2} F \, d x+\frac{\|b\|_{L^{4}}^{4}}{4(2 \mu+\lambda)}\bigg)+ \frac{1}{2}\|\sqrt{\rho} \dot{u}\|_{L^{2}}^{2} \notag\\
&\leq C \int  \rho|u|^{2}|\nabla u|^{2}\,  dx+\int b \cdot \nabla b \cdot u_{t} \, d x-\frac{1}{2} \int \nabla|b|^{2} \cdot u_{t}\,  d x \notag\\
&\quad-\frac{1}{2(2 \mu+\lambda)} \int\Big[\big(\mathcal{Q}(\nabla u)-P \operatorname{div} u+|\operatorname{curl} b|^{2}\big)|b|^{2}+( P u- \nabla \theta) \cdot \nabla|b|^{2}\Big] d x\notag\\
&\quad-\frac{1}{2 \mu+\lambda} \int\Big[\big(\mathcal{Q}(\nabla u)-P \operatorname{div} u+|\operatorname{curl} b|^{2}\big) F+( P u- \nabla \theta) \cdot \nabla F\Big] d x\triangleq\sum_{i=1}^{5} M_{i}.
\end{align}

By H\"{o}lder's, Young's, Gagliardo--Nirenberg inequalities, \eqref{eq1}, and \eqref{eq3}, we get that
\begin{align}
	M_{1} \leq& C\|\rho\|_{L^{3}}\|u\|_{L^{6}}^{2}\|\nabla u\|_{L^{6}}^{2} \notag\\
     \leq& C \bar{\rho}\|\rho\|_{L^{3}}\|\nabla u\|_{L^{2}}^{2}\|\sqrt{\rho} \dot{u}\|_{L^{2}}^{2}+C \bar{\rho}^{2}\|\rho\|_{L^{3}}\|\nabla u\|_{L^{2}}^{2} \|\nabla \theta\|_{L^{2}}^{2}+C\|\rho\|_{L^{3}}\|\nabla u\|_{L^{2}}^{2}\||b||\nabla b |\|_{L^{2}}^{2},\label{xxxx4}\\
	M_{2}= & -\frac{d}{d t} \int b \cdot \nabla u \cdot b \, d x+\int b_{t} \cdot \nabla u \cdot b \, d x+\int b \cdot \nabla u \cdot b_{t} \, d x \notag\\
	\leq & -\frac{d}{d t} \int b \cdot \nabla u \cdot b \, d x+C\|b\|_{L^{3}}\|b_{t}\|_{L^{2}}\|\nabla u\|_{L^{6}} \notag\\
	\leq & -\frac{d}{d t} \int b \cdot \nabla u \cdot b \, d x+C \|b\|_{L^{2}}^{\frac{1}{2}}\|\nabla b\|_{L^{2}}^{\frac{1}{2}}\|b_{t}\|_{L^{2}}\big(\bar{\rho}^{\frac{1}{2}}\|\sqrt{\rho} \dot{u}\|_{L^{2}}+\bar{\rho}\|\nabla \theta\|_{L^{2}}+\||b||  \nabla b | \|_{L^{2}}\big) \notag\\
	\leq & -\frac{d}{d t} \int b \cdot \nabla u \cdot b \, d x +C\bar{\rho}^{\frac{1}{2}}\|b\|_{L^{2}}^{\frac{1}{2}}\|\nabla b\|_{L^{2}}^{\frac{1}{2}}\|b_{t}\|_{L^{2}}\|\sqrt{\rho} \dot{u}\|_{L^{2}}+C\bar{\rho}\|b\|_{L^{2}}^{\frac{1}{2}}\|\nabla b\|_{L^{2}}^{\frac{1}{2}}\|b_{t}\|_{L^{2}}\|\nabla \theta\|_{L^{2}}\notag\\
	& +C\|b\|_{L^{2}}^{\frac{1}{2}}\|\nabla b\|_{L^{2}}^{\frac{1}{2}}\|b_{t}\|_{L^{2}}\||b||\nabla b|\|_{L^{2}}\label{xxxx5},\\
	M_{3}\leq & \frac{1}{2}\frac{d}{d t} \int |b|^{2} \operatorname{div}u \, d x +C\bar{\rho}^{\frac{1}{2}}\|b\|_{L^{2}}^{\frac{1}{2}}\|\nabla b\|_{L^{2}}^{\frac{1}{2}}\|b_{t}\|_{L^{2}}\|\sqrt{\rho} \dot{u}\|_{L^{2}}+C\bar{\rho}\|b\|_{L^{2}}^{\frac{1}{2}}\|\nabla b\|_{L^{2}}^{\frac{1}{2}}\|b_{t}\|_{L^{2}}\|\nabla \theta\|_{L^{2}}\notag\\
	& +C\|b\|_{L^{2}}^{\frac{1}{2}}\|\nabla b\|_{L^{2}}^{\frac{1}{2}}\|b_{t}\|_{L^{2}}\||b||\nabla b|\|_{L^{2}},\label{xxxx6}\\
	M_{4} \leq& C\|\nabla u\|_{L^{2}}\|\nabla u\|_{L^{6}}\|b^{2}\|_{L^{3}}+C\|\nabla u\|_{L^{2}}\|\rho \theta\|_{L^{6}}\|b^{2}\|_{L^{3}}+C\|\nabla b\|_{L^{2}}\|\nabla b\|_{L^{6}}\|b^{2}\|_{L^{3}}\notag\\
	& +C \|\rho\|_{L^{6}}\| u\|_{L^{6}}\|\theta\|_{L^{6}}\|\nabla(|b|^{2})\|_{L^{2}}+C \|\nabla \theta\|_{L^{2}}\|\nabla(|b|^{2})\|_{L^{2}}\notag\\
	 \leq& C\Big(\bar{\rho}^{\frac{1}{2}} \|\nabla u\|_{L^{2}}\|\sqrt{\rho} \dot{u}\|_{L^{2}}\|\nabla b\|_{L^{2}}^{2}+\bar{\rho}\|\nabla u\|_{L^{2}}\|\nabla \theta\|_{L^{2}}\|\nabla b\|_{L^{2}}^{2}+\bar{\rho}^{\frac{1}{2}}\| \rho \|_{L^{3}}^{\frac{1}{2}} \|\nabla u\|_{L^{2}}\|\nabla \theta\|_{L^{2}}\||b||\nabla b | \|_{L^{2}}\Big)\notag\\
	& +C \big(\|\nabla u\|_{L^{2}}\||b|| \nabla b |\|_{L^{2}}\|\nabla b\|_{L^{2}}^{2}+ \|\nabla \theta\|_{L^{2}}\| b \|_{L^{2}}^{\frac{1}{2}} \| \nabla b \|_{L^{2}}^{\frac{1}{2}} \| \nabla^{2} b \|_{L^{2}}+  \|b\|_{L^{2}} \|\nabla b\|_{L^{2}} \|\nabla^{2} b\|_{L^{2}}^{2}\big)\label{xxxx7},\\
	M_{5}  \leq& C\Big(\|\nabla F\|_{L^{2}}\|\nabla u\|_{L^{\frac{12}{5}}}^{2}+ \|\nabla u\|_{L^{2}}\|\rho \theta\|_{L^{3}}\|\nabla F\|_{L^{2}}+ \| b\|_{L^{3}}\|\nabla^{2} b\|_{L^{2}}\|\nabla F\|_{L^{2}}\Big)\notag\\
	&  +C\big( \|\rho\|_{L^{6}}\| u\|_{L^{6}}\|\theta\|_{L^{6}}\|\nabla F\|_{L^{2}}+  \|\nabla \theta\|_{L^{2}}\|\nabla F\|_{L^{2}}\big)\notag\\
	 \leq& C \Big(\bar{\rho}^{\frac{1}{2}}\|\nabla \theta\|_{L^{2}} \|\sqrt{\rho} \dot{u}\|_{L^{2}}+ \|\nabla \theta\|_{L^{2}}\| b \|_{L^{2}}^{\frac{1}{2}} \| \nabla b \|_{L^{2}}^{\frac{1}{2}} \| \nabla^{2} b \|_{L^{2}}+ \bar{\rho}^{\frac{1}{2}} \|\sqrt{\rho} \dot{u}\|_{L^{2}}  \| b\|_{L^{2}}^{\frac{1}{2}}\| \nabla b\|_{L^{2}}^{\frac{1}{2}}\| \nabla^{2}b\|_{L^{2}}\Big)\notag\\
	&  +C\Big(\| b\|_{L^{2}}^{\frac{1}{2}}\| \nabla b\|_{L^{2}}^{\frac{1}{2}}\| \nabla^{2}b\|_{L^{2}}\||b||\nabla b | \|_{L^{2}}+ \bar{\rho}^{\frac{1}{2}} \| \rho\|_{L^{3}}^{\frac{1}{2}}\| \nabla u\|_{L^{2}}\| \nabla \theta\|_{L^{2}} \||b||\nabla b | \|_{L^{2}}\Big)\notag\\
	&  +C\Big(\bar{\rho}^{\frac{1}{2}}\| \rho\|_{L^{3}}^{\frac{1}{2}}\| \nabla u\|_{L^{2}}\|\sqrt{\rho} \dot{u}\|_{L^{2}}^{2}+ \| \rho\|_{L^{3}}^{\frac{1}{2}}\| \nabla u\|_{L^{2}}\|\sqrt{\rho} \dot{u}\|_{L^{2}}\||b||\nabla b | \|_{L^{2}}+ \bar{\rho}^{\frac{1}{2}}\| \nabla u\|_{L^{2}}\|\sqrt{\rho} \dot{u}\|_{L^{2}}\| \nabla b\|_{L^{2}}^{2}\Big)\notag\\
	&  +C\Big(\| \nabla u\|_{L^{2}} \| \nabla b\|_{L^{2}}^{2}\||b||\nabla b | \|_{L^{2}}+ \bar{\rho} \| \rho\|_{L^{3}}^{\frac{1}{2}}\| \nabla u\|_{L^{2}}\| \nabla \theta\|_{L^{2}}  \|\sqrt{\rho} \dot{u}\|_{L^{2}}\Big).\label{xxxx8}
\end{align}

Hence, putting \eqref{xxxx4}--\eqref{xxxx8} into \eqref{xxxx3} and using Cauchy--Schwarz inequality, we derive that
\begin{align}\label{xxxx9}
&\frac{d}{d t}  \mathcal{E}(t) +\|\sqrt{\rho} \dot{u}\|_{L^{2}}^{2}\notag\\
&\leq C \big(\bar{\rho}\|\rho\|_{L^{3}}\|\nabla u\|_{L^{2}}^{2}\|\sqrt{\rho} \dot{u}\|_{L^{2}}^{2}+  \bar{\rho}^{2}\|\rho\|_{L^{3}}\|\nabla u\|_{L^{2}}^{2} \|\nabla \theta\|_{L^{2}}^{2}+ \|\rho\|_{L^{3}}\|\nabla u\|_{L^{2}}^{2}\||b||\nabla b |\|_{L^{2}}^{2}\big)\notag\\
&\quad +C\Big(\bar{\rho}^{\frac{1}{2}}\|b\|_{L^{2}}^{\frac{1}{2}}\|\nabla b\|_{L^{2}}^{\frac{1}{2}}\|b_{t}\|_{L^{2}}\|\sqrt{\rho} \dot{u}\|_{L^{2}}+\bar{\rho}\|b\|_{L^{2}}^{\frac{1}{2}}\|\nabla b\|_{L^{2}}^{\frac{1}{2}}\|b_{t}\|_{L^{2}}\|\nabla \theta\|_{L^{2}}+ \|b\|_{L^{2}}^{\frac{1}{2}}\|\nabla b\|_{L^{2}}^{\frac{1}{2}}\|b_{t}\|_{L^{2}}\||b||\nabla b|\|_{L^{2}}\Big)\notag\\
&\quad +C\Big(\bar{\rho}^{\frac{1}{2}} \|\nabla u\|_{L^{2}}\|\sqrt{\rho} \dot{u}\|_{L^{2}}\|\nabla b\|_{L^{2}}^{2}+\bar{\rho}\|\nabla u\|_{L^{2}}\|\nabla \theta\|_{L^{2}}\|\nabla b\|_{L^{2}}^{2}+\|\nabla u\|_{L^{2}}\||b|| \nabla b |\|_{L^{2}}\|\nabla b\|_{L^{2}}^{2}\Big)\notag\\
&\quad+C\Big(\bar{\rho}^{\frac{1}{2}}\| \rho \|_{L^{3}}^{\frac{1}{2}} \|\nabla u\|_{L^{2}}\|\nabla \theta\|_{L^{2}}\||b||\nabla b | \|_{L^{2}}+ \|\nabla \theta\|_{L^{2}}\| b \|_{L^{2}}^{\frac{1}{2}} \| \nabla b \|_{L^{2}}^{\frac{1}{2}} \| \nabla^{2} b \|_{L^{2}}+  \|b\|_{L^{2}} \|\nabla b\|_{L^{2}} \|\nabla^{2} b\|_{L^{2}}^{2}\Big)  \notag\\
&\quad+C\Big(\bar{\rho}^{\frac{1}{2}}\|\nabla \theta\|_{L^{2}} \|\sqrt{\rho} \dot{u}\|_{L^{2}}+ \bar{\rho}^{\frac{1}{2}} \|\sqrt{\rho} \dot{u}\|_{L^{2}}  \| b\|_{L^{2}}^{\frac{1}{2}}\| \nabla b\|_{L^{2}}^{\frac{1}{2}}\| \nabla^{2}b\|_{L^{2}}+ \| b\|_{L^{2}}^{\frac{1}{2}}\| \nabla b\|_{L^{2}}^{\frac{1}{2}}\| \nabla^{2}b\|_{L^{2}}\||b||\nabla b | \|_{L^{2}}\Big)\notag\\
&\quad +C\Big(\bar{\rho}^{\frac{1}{2}}\| \rho\|_{L^{3}}^{\frac{1}{2}}\| \nabla u\|_{L^{2}}\|\sqrt{\rho} \dot{u}\|_{L^{2}}^{2}+ \| \rho\|_{L^{3}}^{\frac{1}{2}}\| \nabla u\|_{L^{2}}\|\sqrt{\rho} \dot{u}\|_{L^{2}}\||b||\nabla b | \|_{L^{2}} +  \bar{\rho} \| \rho\|_{L^{3}}^{\frac{1}{2}}\| \nabla u\|_{L^{2}}\| \nabla \theta\|_{L^{2}}  \|\sqrt{\rho} \dot{u}\|_{L^{2}}\Big)\notag\\
&\leq \Big(C \bar{\rho}\|\rho\|_{L^{3}}\|\nabla u\|_{L^{2}}^{2}+C \bar{\rho}^{\frac{1}{2}}\|\rho\|_{L^{3}}^{\frac{1}{2}} \|\nabla u\|_{L^{2}}+\frac{1}{16}\Big)\|\sqrt{\rho} \dot{u}\|_{L^{2}}^{2}+CK_{1}(t)\|\nabla \theta\|_{L^{2}}^{2}+CK_{2}(t),
\end{align}
where
\begin{align*}
\mathcal{E}(t)\triangleq\mu\|\operatorname{curl} u\|_{L^{2}}^{2}+\frac{\|F\|_{L^{2}}^{2}}{2 \mu+\lambda}+\frac{1}{2 \mu+\lambda} \int|b|^{2} F \, d x+\frac{\|b\|_{L^{4}}^{4}}{4(2 \mu+\lambda)} +\int b \cdot \nabla u \cdot b \, d x-\frac{1}{2} \int|b|^{2} \operatorname{div} u \, d x.
\end{align*}
Choosing $\epsilon$ in \eqref{p1} sufficiently small such that
\begin{align}\label{sm4}
C\max\Big\{ \bar{\rho}\|\rho(t)\|_{L^{3}}\|\nabla u(t)\|_{L^{2}}^{2},~~\bar{\rho}^{\frac{1}{2}}\|\rho(t)\|_{L^{3}}^{\frac{1}{2}} \|\nabla u(t)\|_{L^{2}}\Big\} \leq \frac{1}{4},
\end{align}
then one infers from \eqref{xxxx9} that
\begin{align}\label{xiu3}
\frac{d}{d t} \mathcal{E}(t)+\|\sqrt{\rho} \dot{u}(t)\|_{L^{2}}^{2} \leq C\left(K_{1}(t)\|\nabla \theta\|_{L^{2}}^{2}+K_{2}(t)\right).
\end{align}
Noting that, for any $\delta>0$,
\begin{align*}
	\left|\int\bigg(\frac{1}{2 \mu+\lambda} |b|^{2} F  +b \cdot \nabla u \cdot b-\frac{1}{2} |b|^{2} \operatorname{div} u\bigg) d x\right|
	&\leq \delta \|\nabla u\|_{L^{2}}^{2}+\delta \|F\|_{L^{2}}^{2}+C\|b\|_{L^{4}}^{4}\\
	&\leq 2\delta \|\nabla u\|_{L^{2}}^{2}
	+C\big(\|b\|_{L^{4}}^{4}+\bar{\rho} \|\sqrt{\rho}\theta\|_{L^{2}}^{2} \big),
\end{align*}
and
\begin{align*}
	\|\nabla u\|_{L^{2}}^{2}
	\leq C\left(\|\operatorname{curl} u\|_{L^{2}}^{2}
	+\|F\|_{L^{2}}^{2}
	+\bar{\rho} \|\sqrt{\rho}\theta\|_{L^{2}}^{2}
	+\|b\|_{L^{4}}^{4}\right).
\end{align*}
Integrating \eqref{xiu3} over $(0,T)$ and choosing $\delta$ sufficiently small, then one gets the desired \eqref{xxxx0}.
\end{proof}

\begin{lemma}\label{lem5}
Under the hypotheses of Proposition \ref{pro},  it holds that, for all $t \in[0, T]$,
\begin{align}\label{xxxxx0}
	&  \|\nabla b(t)\|_{L^{2}}^{2}+\|b(t)\|_{L^{4}}^{4} +\int_{0}^{T}\big(\|b_{t}\|_{L^{2}}^{2}+\|\nabla^{2} b\|_{L^{2}}^{2}+\|| b|  |\nabla b| \|_{L^{2}}^{2}\big) d t \notag\\
	&\leq C\big(\|\nabla b_{0}\|_{L^{2}}^{2}+\|b_{0}\|_{L^{4}}^{4}\big)+ C\sup _{t \in[0, T]}(\|b(t)\|_{L^{2}}^{2}\|\nabla u(t)\|_{L^{2}}^{2}) \sup _{t \in[0, T]}\|\nabla u(t)\|_{L^{2}}^{4}\int_{0}^{T}\|\nabla u\|_{L^{2}}^{2}\, d t.
\end{align}
\end{lemma}

\begin{proof}
According to \eqref{1}$_{4}$, $\operatorname{curl}(u \times b)=b \cdot \nabla u-u \cdot \nabla b-\operatorname{div} u \cdot b$, and Gagliardo--Nirenberg inequality, one gets that
\begin{align}\label{xxxxx1}
	\frac{d}{d t}\|\nabla b\|_{L^{2}}^{2}+\|b_{t}\|_{L^{2}}^{2}+\|\nabla^{2} b\|_{L^{2}}^{2} & =\int\left|b_{t}-\Delta b\right|^{2} d x=\int|b \cdot \nabla u-u \cdot \nabla b-b \operatorname{div} u|^{2} \, d x \notag\\
	& \leq C\|\nabla u\|_{L^{2}}^{2}\|b\|_{L^{\infty}}^{2}+C\|u\|_{L^{6}}^{2}\|\nabla b\|_{L^{3}}^{2} \notag\\
	&\leq C\|b\|_{L^{2}}^{\frac{1}{2}} \|\nabla^{2} b \|_{L^{2}}^{\frac{3}{2}}\|\nabla u\|_{L^{2}}^{2} \notag\\
	&\leq \frac{1}{8} \|\nabla^{2} b \|_{L^{2}}^{2}+C\left(\|b\|_{L^{2}}^{2}\|\nabla u\|_{L^{2}}^{2}\right)\|\nabla u\|_{L^{2}}^{6}.
\end{align}
Multiplying \eqref{1}$_{4}$ by $4|b|^{2} b$ and integration by parts, we use Sobolev's inequality to obtain that
\begin{align*}
	\frac{d}{d t} \|b\|_{L^{4}}^{4}+4 \int\big(|\nabla b|^{2}|b|^{2}+2|\nabla|b||^{2}|b|^{2}\big) \, d x &=4 \int b \cdot \nabla u \cdot b|b|^{2} d x-3 \int|b|^{4} \operatorname{div} u \, d x \\
	&\leq C \int|b|^{2}|\nabla u|^{2} \, d x+C \int|b|^{6} \, d x \\
	&\leq C\|b\|_{L^{\infty}}^{2}\|\nabla u\|_{L^{2}}^{2}+C\|b\|_{L^{6}}^{2}\||b|^{2}\|_{L^{6}}\||b|^{2}\|_{L^{2}} \\
	&\leq C\|b\|_{L^{2}}^{\frac{1}{2}}\|\nabla^{2} b\|_{L^{2}}^{\frac{3}{2}}\|\nabla u\|_{L^{2}}^{2}+C\||b||\nabla b|\|_{L^{2}}\|b\|_{L^{4}}^{2}\|\nabla b\|_{L^{2}}^{2}.
\end{align*}
The last term can be estimated by Gagliardo--Nirenberg and Cauchy--Schwarz inequalities,
\begin{align*}
	C\| |b||\nabla b|\|_{L^{2}}\|b\|_{L^{4}}^{2}\|\nabla b\|_{L^{2}}^{2} & \leq C \| |b||\nabla b |\|_{L^{2}}\|b\|_{L^{3}}\|\nabla b\|_{L^{2}}\|b\|_{L^{2}} \|\nabla^{2} b \|_{L^{2}} \\
	& \leq C\| |b||\nabla b|\|_{L^{2}} \| b \|_{L^{2}}^{\frac{3}{2}} \| \nabla b \|_{L^{2}}^{\frac{3}{2}}\| \nabla^{2} b \|_{L^{2}} \\
	& \leq \frac{1}{2}\||b||\nabla b|\|_{L^{2}}^{2}+C \| b \|_{L^{2}}^{3} \| \nabla b \|_{L^{2}}^{3} \| \nabla^{2} b \|_{L^{2}}^{2}.
\end{align*}
Thus, one has that
\begin{align}\label{xiu2}
	\frac{d}{d t}\|b\|_{L^{4}}^{4}+\||b||\nabla b|\|_{L^{2}}^{2} \leq \frac{1}{8} \|\nabla^{2} b \|_{L^{2}}^{2}+C\big(\|b\|_{L^{2}}^{2}\|\nabla b\|_{L^{2}}^{2}\big)^{\frac{3}{2}} \|\nabla^{2} b \|_{L^{2}}^{2} +C\big(\|b\|_{L^{2}}^{2}\|\nabla u\|_{L^{2}}^{2}\big)\|\nabla u\|_{L^{2}}^{6}.
\end{align}
Combining \eqref{xxxxx1} and \eqref{xiu2}, and using the {\it a priori} assumption \eqref{p1} to choose $\epsilon$  sufficiently small such that for any $t$
\begin{align}\label{sm5}
	C\big(\|b(t)\|_{L^{2}}^{2}\|\nabla b(t)\|_{L^{2}}^{2}\big)^{\frac{3}{2}} \leq \frac{1}{4},
\end{align}
one arrives at
\begin{align}\label{xxxxx2}
	\frac{d}{d t}\big(\|\nabla b\|_{L^{2}}^{2}+\|b\|_{L^{4}}^{4}\big)+\|b_{t}\|_{L^{2}}^{2}+\|\nabla^{2} b\|_{L^{2}}^{2}+\||b||\nabla b|\|_{L^{2}}^{2} \leq C\big(\|b\|_{L^{2}}^{2}\|\nabla u\|_{L^{2}}^{2}\big)\|\nabla u\|_{L^{2}}^{6}.
\end{align}
Consequently, integrating \eqref{xxxxx2} in $t$ over $(0, T)$ leads to \eqref{xxxxx0}.
\end{proof}

\begin{corollary}\label{co1}
Under the hypotheses of Proposition \ref{pro}, it holds that, for all $t \in[0, T]$,
\begin{align}\label{zui}
S_{t}+( \bar{\rho}^{2}+\bar{\rho}) \sup _{t \in[0, T]}\|\rho(t)\|_{L^{3}}\int_{0}^{T}\|\nabla \theta\|_{L^{2}}^{2} d t   \leq \frac{3}{2} \epsilon.
\end{align}
\end{corollary}

\begin{proof}
Multiplying \eqref{x0} by $2C(\bar{\rho}^{2}+\bar{\rho})$ and adding the resultant to \eqref{xx1} yields that
\begin{align}\label{c1}
& \|\rho(t)\|_{L^{3}}+C( \bar{\rho}^{2}+\bar{\rho}) ( \| \sqrt{\rho}u(t)\|_{L^{2}}^{2}+\| b(t)\|_{L^{2}}^{2} )  +C( \bar{\rho}^{2}+\bar{\rho})\int_{0}^{T} (  \| \nabla u\|_{L^{2}}^{2}+\| \nabla b\|_{L^{2}}^{2} ) \, d t \notag\\
&\leq C   \|\rho_{0}\|_{L^{3}}+C( \bar{\rho}^{2}+\bar{\rho}) ( \| \sqrt{\rho_{0}}u_{0}\|_{L^{2}}^{2}+\| b_{0}\|_{L^{2}}^{2} ) \notag \\
& \quad+C(\bar{\rho}^{2}+\bar{\rho})\sup _{ t\in[0, T]}\|\rho(t)\|_{L^{3}}\int_{0}^{T}\|\nabla \theta\|_{L^{2}}^{2} \, d t\, \sup _{ t\in[0, T]}\|\rho(t)\|_{L^{3}}.
\end{align}
Choosing $\epsilon$ in \eqref{p1} sufficiently small such that
\begin{align}\label{sm6}
C(\bar{\rho}^{2}+\bar{\rho})\sup _{t \in[0, T]}\|\rho(t)\|_{L^{3}} \int_{0}^{T}\|\nabla \theta\|_{L^{2}}^{2} d t
\leq \frac{1}{2},
\end{align}
then one deduces from \eqref{c1} that
\begin{align}\label{c3}
	&  \|\rho(t)\|_{L^{3}}+( \bar{\rho}^{2}+\bar{\rho})\big( \| \sqrt{\rho}u(t)\|_{L^{2}}^{2}+\| b(t)\|_{L^{2}}^{2}\big)  +( \bar{\rho}^{2}+\bar{\rho})\int_{0}^{T}\big(  \| \nabla u\|_{L^{2}}^{2}+\| \nabla b\|_{L^{2}}^{2}\big)\, d t \notag\\
	& \leq C  \big(\|\rho_{0}\|_{L^{3}}+( \bar{\rho}^{2}+\bar{\rho}) ( \| \sqrt{\rho_{0}}u_{0}\|_{L^{2}}^{2}+\| b_{0}\|_{L^{2}}^{2} )\big)= C S_{0}^{\prime}.
\end{align}

It follows from \eqref{xxxxx0} and \eqref{xxx1} that
\begin{align}\label{xxx11}
	&  \|\sqrt\rho \theta(t)\|_{L^{2}}^{2}+\|\nabla b(t)\|_{L^{2}}^{2}+\|b(t)\|_{L^{4}}^{4} +\int_{0}^{T}\big( \|b_{t} \|_{L^{2}}^{2}+ \|\nabla^{2} b \|_{L^{2}}^{2}+\||b||\nabla b| \|_{L^{2}}^{2}+\|\nabla \theta\|_{L^{2}}^{2}\big) d t \notag\\
	&\leq C\big( \|\sqrt{\rho_{0}}\theta_{0}\|_{L^{2}}^{2}+\|\nabla b_{0} \|_{L^{2}}^{2}+ \|b_{0} \|_{L^{4}}^{4}\big)+C \sup _{t \in[0, T]}\big(\|\nabla u(t)\|_{L^{2}}^{2}\|\rho(t)\|_{L^{3}}\big) \int_{0}^{T}\|\sqrt{\rho} \dot{u}\|_{L^{2}}^{2} d t\notag\\
	&\quad +C \sup _{t \in[0, T]}\big(\|b(t)\|_{L^{2}} \|\nabla b(t)\|_{L^{2}}\big)\int_{0}^{T}\|\nabla^{2} b\|_{L^{2}}^{2} d t +C  \int_{0}^{T}\|\nabla u\|_{L^{2}}^{2}\|\nabla b\|_{L^{2}}^{4}\, d t\notag\\
	&\quad + C\sup _{t \in[0, T]}(\|b(t)\|_{L^{2}}^{2}\|\nabla u(t)\|_{L^{2}}^{2}) \sup _{t \in[0, T]}\|\nabla u(t)\|_{L^{2}}^{4}\int_{0}^{T}\|\nabla u\|_{L^{2}}^{2}\, d t.
\end{align}
Choosing $\epsilon$ in \eqref{p1} sufficiently small such that, for any $t\in[0,T]$,
\begin{align}\label{sm11}
C\|b(t)\|_{L^{2}} \|\nabla b(t)\|_{L^{2}}\le \frac{1}{16},
\end{align}
one thus infers from \eqref{xxx11} that
\begin{align}\label{xxx12}
& \|\sqrt\rho \theta(t)\|_{L^{2}}^{2}+\|\nabla b(t)\|_{L^{2}}^{2}+\|b(t)\|_{L^{4}}^{4} +\int_{0}^{T}\big( \|b_{t} \|_{L^{2}}^{2}+ \|\nabla^{2} b \|_{L^{2}}^{2}+\||b||\nabla b| \|_{L^{2}}^{2}+\|\nabla \theta\|_{L^{2}}^{2}\big) d t \notag\\
&\leq  C\big(\|\sqrt{\rho_{0}} \theta_{0}\|_{L^{2}}^{2}+\|\nabla b_{0} \|_{L^{2}}^{2}+ \|b_{0} \|_{L^{4}}^{4}\big)+C   \sup _{t \in[0, T]}\big(\|\nabla u(t)\|_{L^{2}}^{2}\|\rho(t)\|_{L^{3}}\big) \int_{0}^{T}\|\sqrt{\rho} \dot{u}\|_{L^{2}}^{2} d t\notag\\
&\quad+ C\sup _{t \in[0, T]}(\|b(t)\|_{L^{2}}^{2}\|\nabla u(t)\|_{L^{2}}^{2}) \sup _{t \in[0, T]}\|\nabla u(t)\|_{L^{2}}^{2}\int_{0}^{T}\|\nabla u\|_{L^{2}}^{2}\, d t\sup _{t \in[0, T]}\|\nabla u(t)\|_{L^{2}}^{2}\notag\\
&\quad +C  \int_{0}^{T}\|\nabla u\|_{L^{2}}^{2}\|\nabla b\|_{L^{2}}^{4}\, d t .
\end{align}
Multiplying \eqref{xxx1} by $\bar{\rho}$ gives that
\begin{align}\label{c5}
	 &\bar{\rho}\|\sqrt\rho \theta(t)\|_{L^{2}}^{2}+\bar{\rho}\int_{0}^{T}\|\nabla \theta\|_{L^{2}}^{2} \, d t\notag\\
	&\leq   C\bar{\rho}\|\sqrt{\rho_{0}} \theta_{0}\|_{L^{2}}^{2}+C\bar{\rho} \sup _{t \in[0, T]}\left(\|\nabla u(t)\|_{L^{2}}^{2}\|\rho(t)\|_{L^{3}}\right) \int_{0}^{T}\|\sqrt{\rho} \dot{u}\|_{L^{2}}^{2}\,  d t\notag\\
	&\quad +C\bar{\rho} \sup _{t \in[0, T]}\big(\|b(t)\|_{L^{2}} \|\nabla b(t)\|_{L^{2}}\big)\int_{0}^{T}\|\nabla^{2} b\|_{L^{2}}^{2} \, d t  +C\bar{\rho}  \int_{0}^{T}\|\nabla u\|_{L^{2}}^{2}\|\nabla b\|_{L^{2}}^{4}\, d t.
\end{align}
Substituting \eqref{xxx12} and \eqref{c5} into \eqref{xxxx0} leads to
\begin{align}\label{c6}
	& \|\nabla u(t)\|_{L^{2}}^{2}+ (\bar{\rho}+1)  \| \sqrt{\rho}\theta(t)\|_{L^{2}}^{2}+\| \nabla b(t)\|_{L^{2}}^{2}+\|  b(t)\|_{L^{4}}^{4} \notag\\
	&\quad +\int_{0}^{T}\left(  \| \sqrt{\rho}\dot{u}\|_{L^{2}}^{2}+  (\bar{\rho}+1)  \| \nabla \theta\|_{L^{2}}^{2}+ \|b_{t} \|_{L^{2}}^{2}+ \|\nabla^{2} b \|_{L^{2}}^{2}+\||b|| \nabla b| \|_{L^{2}}^{2}\right) d t \notag\\
	&\leq C \left(\|\nabla u_{0}\|_{L^{2}}^{2}+ (\bar{\rho}+1)  \| \sqrt{\rho_{0}}\theta_{0}\|_{L^{2}}^{2}+\| \nabla b_{0}\|_{L^{2}}^{2}+\|  b_{0}\|_{L^{4}}^{4}\right)\notag\\
	&\quad +C (\bar{\rho}+1)  \sup _{t \in[0, T]}\big(\|\nabla u(t)\|_{L^{2}}^{2}\|\rho(t)\|_{L^{3}}\big) \int_{0}^{T}\|\sqrt{\rho} \dot{u}\|_{L^{2}}^{2} \, d t\notag\\
	&\quad + C\sup _{t \in[0, T]}(\|b(t)\|_{L^{2}}^{2}\|\nabla u(t)\|_{L^{2}}^{2}) \sup _{t \in[0, T]}\|\nabla u(t)\|_{L^{2}}^{2}\int_{0}^{T}\|\nabla u\|_{L^{2}}^{2}\, d t\sup _{t \in[0, T]}\|\nabla u(t)\|_{L^{2}}^{2}\notag\\
	 &\quad +C (\bar{\rho}+1) \int_{0}^{T}\|\nabla u\|_{L^{2}}^{2}\|\nabla b\|_{L^{2}}^{4}\, d t+C \sup _{t \in[0, T]} K_{1}(t)  \int_{0}^{T}\|\nabla \theta\|_{L^{2}}^{2} d t \notag\\
	 &\quad  +C\bar{\rho} \sup _{t \in[0, T]}\big(\|b(t)\|_{L^{2}} \|\nabla b(t)\|_{L^{2}}\big)\int_{0}^{T}\|\nabla^{2} b\|_{L^{2}}^{2} \, d t+C\int_{0}^{T} K_{2}(t) d t\notag\\
	&\leq C  \left(\|\nabla u_{0}\|_{L^{2}}^{2}+(\bar{\rho}+1) \| \sqrt{\rho_{0}}\theta_{0}\|_{L^{2}}^{2}+\| \nabla b_{0}\|_{L^{2}}^{2}+\|  b_{0}\|_{L^{4}}^{4}\right)+C\int_{0}^{T} K_{2}(t) d t\notag\\
	&\quad +C (\bar{\rho}+1) \int_{0}^{T}\|\nabla u\|_{L^{2}}^{2}\|\nabla b\|_{L^{2}}^{4}\, d t,
\end{align}
where $\epsilon$ in \eqref{p1} is so small that, for any $t\in[0,T]$,
\begin{align}\label{sm7}
C\max\big\{&(\bar{\rho}+1)  \|\nabla u(t)\|_{L^{2}}^{2}\|\rho(t)\|_{L^{3}},~\bar{\rho}\|b(t)\|_{L^{2}}\|\nabla b(t)\|_{L^{2}},~\|  b(t)\|_{L^{2}}^{2}\|\nabla u(t)\|_{L^{2}}^{2},\notag\\
&\sup _{t \in[0, T]}\|\nabla u(t)\|_{L^{2}}^{2}\int_{0}^{T}\|\nabla u\|_{L^{2}}^{2}\, d t\big\}\leq \frac{1}{16}.
\end{align}
In addition, noting that
\begin{align*}
C (\bar{\rho}+1) \int_{0}^{T}\|\nabla u\|_{L^{2}}^{2}\|\nabla b\|_{L^{2}}^{4}\, d t\leq C (\bar{\rho}+1)\sup _{t \in[0, T]} \|\nabla b(t)\|_{L^{2}}^{2} \int_{0}^{T}\|\nabla u\|_{L^{2}}^{2}\, d t \sup _{t \in[0, T]} \|\nabla b(t)\|_{L^{2}}^{2},
\end{align*}
and
\begin{align*}
	&\int_{0}^{T} K_{2}(t)\, dt\\
	&=C\int_{0}^{T} \Big(  \bar{\rho}\|b\|_{L^{2}}\|\nabla b\|_{L^{2}}\|b_{t}\|_{L^{2}}^{2} + \|\rho\|_{L^{3}}\|\nabla u\|_{L^{2}}^{2}\||b|| \nabla b |\|_{L^{2}}^{2}+  \bar{\rho}\|b\|_{L^{2}}\|\nabla b\|_{L^{2}}\|\nabla^{2} b\|_{L^{2}}^{2} +  \bar{\rho}  \|\nabla u\|_{L^{2}}^{2}\|\nabla b\|_{L^{2}}^{4}  \\
	&\quad +  \|\rho\|_{L^{3}}\|\nabla u\|_{L^{2}}^{2}\||b||\nabla b |\|_{L^{2}}^{2}+ \bar{\rho}\|b\|_{L^{2}}^{\frac{1}{2}}\|\nabla b\|_{L^{2}}^{\frac{1}{2}}\|b_{t}\|_{L^{2}}\|\nabla \theta\|_{L^{2}}+ \|b\|_{L^{2}}^{\frac{1}{2}}\|\nabla b\|_{L^{2}}^{\frac{1}{2}}\|b_{t}\|_{L^{2}}\||b||\nabla b|\|_{L^{2}}   \\
	&\quad +  \|\nabla u\|_{L^{2}}\|\nabla b\|_{L^{2}}^{2}\||b|| \nabla b|  \|_{L^{2}}+ \bar{\rho}\|\nabla u\|_{L^{2}}\|\nabla \theta\|_{L^{2}}\|\nabla b\|_{L^{2}}^{2}+\| b\|_{L^{2}}^{\frac{1}{2}}\| \nabla b\|_{L^{2}}^{\frac{1}{2}}\| \nabla^{2}b\|_{L^{2}}\||b||\nabla b | \|_{L^{2}}   \\
	&\quad +   \bar{\rho}^{\frac{1}{2}}\| \rho \|_{L^{3}}^{\frac{1}{2}} \|\nabla u\|_{L^{2}}\|\nabla \theta\|_{L^{2}}\||b||\nabla b | \|_{L^{2}}+\|\nabla \theta\|_{L^{2}}\| b \|_{L^{2}}^{\frac{1}{2}} \| \nabla b \|_{L^{2}}^{\frac{1}{2}} \| \nabla^{2} b \|_{L^{2}}+   \|b\|_{L^{2}} \|\nabla b\|_{L^{2}} \|\nabla^{2} b\|_{L^{2}}^{2}   \Big)\, dt
\end{align*}
can be absorbed by the left-hand side of \eqref{c6} provided $\epsilon$ in \eqref{p1} is chosen so small that, for any $t\in[0,T]$,
\begin{align}\label{sm9}
C\max\Big\{ \|\rho(t) \|_{L^{3}} \|\nabla u(t)\|_{L^{2}}^{2},~(\bar{\rho}+1)  \|\nabla b(t)\|_{L^{2}} \|b(t)\|_{L^{2}},~(\bar{\rho}+1)\sup _{t \in[0, T]} \|\nabla b(t)\|_{L^{2}}^{2} \int_{0}^{T}\|\nabla u\|_{L^{2}}^{2}\, d t\Big\}  \leq \frac{1}{16}.
\end{align}
Consequently, one obtains that
\begin{align}\label{c9}
&\|\nabla u(t)\|_{L^{2}}^{2}+(\bar{\rho}+1) \| \sqrt{\rho}\theta(t)\|_{L^{2}}^{2}+\| \nabla b(t)\|_{L^{2}}^{2}+\|  b(t)\|_{L^{4}}^{4}\notag\\ &\quad+\int_{0}^{T}\big(  \| \sqrt{\rho}\dot{u}\|_{L^{2}}^{2}+ (\bar{\rho}+1)\| \nabla \theta\|_{L^{2}}^{2}+ \|b_{t} \|_{L^{2}}^{2}+ \|\nabla^{2} b \|_{L^{2}}^{2}+\||b||\nabla b| \|_{L^{2}}^{2}\big) \, d t \notag\\
&  \leq C  \big(\|\nabla u_{0}\|_{L^{2}}^{2}+(\bar{\rho}+1) \| \sqrt{\rho_{0}}\theta_{0}\|_{L^{2}}^{2}+\| \nabla b_{0}\|_{L^{2}}^{2}+\|  b_{0}\|_{L^{4}}^{4}\big)= C S_{0}^{\prime\prime}.
\end{align}
Multiplying \eqref{c9} by $\big(1+\bar{\rho}+\frac{1}{\bar{\rho}}\big)$ and using \eqref{c3}, we arrive at
\begin{align*}
&\bigg(1+\bar{\rho}+\frac{1}{\bar{\rho}}\bigg)S_{t}^{1} S_{t}^{2}+( \bar{\rho}^{2}+\bar{\rho})\sup _{t \in[0, T]}\|\rho(t)\|_{L^{3}} \int_{0}^{T}\|\nabla \theta\|_{L^{2}}^{2} d t\notag\\
&\quad +(\bar{\rho}+1)\sup _{t \in[0, T]} \|\nabla b(t)\|_{L^{2}}^{2} \int_{0}^{T}\|\nabla u\|_{L^{2}}^{2}\, d t
+\sup _{t \in[0, T]}\|\nabla u(t)\|_{L^{2}}^{2}\int_{0}^{T}\|\nabla u\|_{L^{2}}^{2}\, d t\notag\\
&\leq C \bigg(1+\bar{\rho}+\frac{1}{\bar{\rho}}\bigg) S_{0}^{\prime} S_{0}^{\prime\prime}= C \epsilon_{0}\leq \frac{3}{2} \epsilon,
\end{align*}
provided $\epsilon \geq \frac{2}{3}C \epsilon_{0}\triangleq C_{1}\epsilon_{0}$. This completes the proof of Corollary \ref{co1}.
\end{proof}

\begin{corollary}\label{co2}
Under the hypotheses of Proposition \ref{pro}, it holds that, for any  $(x, t) \in \mathbb{R}^{3} \times[0, T]$,
\begin{align*}
\rho(x, t) \leq \frac{3}{2} \bar{\rho}.
\end{align*}

\end{corollary}
\begin{proof}
Fix $(x,t) \in \mathbb{R}^{3} \times [0,T]$ and a constant $\delta>0$, we set
\begin{align*}
\rho^{\delta}(y, s)=\rho(y, s)+\delta \exp \left\{-\int_{0}^{s} \operatorname{div} u(X(\tau ; x, t), \tau) d \tau\right\}>0,
\end{align*}
where  $X(s ; x, t)$  is the flow map defined by
\begin{align*}
	\begin{cases}
	\frac{d}{d s} X(s ; x, t)=u(X(s ; x, t), s), ~~0 \leq s<t, \\
	X(t ; x, t)=x.
\end{cases}
\end{align*}
Using the fact that
\begin{align*}
\frac{d}{d s} (f(X(s ; x, t), s)= (f_{s}+u \cdot \nabla f )(X(s ; x, t), s),
\end{align*}
one obtains from \eqref{1}$_{1}$ that
\begin{align*}
\frac{d}{d s} \rho^{\delta}(X(s ; x, t), s)+\rho^{\delta}(X(s ; x, t), s) \operatorname{div} u(X(s ; x, t), s)=0.
\end{align*}
This implies that
\begin{align}\label{xu1}
Y^{\prime}(s)=g(s)+b^{\prime}(s),
\end{align}
where
\begin{gather}
	Y(s)=\ln \rho^{\delta}(X(s ; x, t), s), ~~ g(s)=-\frac{P(X(s ; x, t), s)}{2 \mu+\lambda},\notag\\
		b(s)=-\frac{1}{2 \mu+\lambda} \int_{0}^{s}\Big(\frac{1}{2}|b(X(\tau ; x, t), \tau)|^{2}+F(X(\tau ; x, t), \tau)\Big) d \tau\label{s1}.
\end{gather}

A direct computation yields that
\begin{align}\label{s2}
F(X(\tau ; x, t), \tau)= -\frac{d}{d \tau}\big[(-\Delta)^{-1} \operatorname{div}(\rho u)\big]+ [u_{i}, R_{i j} ] (\rho u_{j} )+(-\Delta)^{-1} \operatorname{div} \operatorname{div}(b \otimes b),
\end{align}
where  $ [u_{i}, R_{i j} ]=u_{i} R_{i j}-R_{i j} u_{i}$ and  $R_{i j}=\partial_{i}(-\Delta)^{-1} \partial_{j}$  denotes the Riesz transform on  $\mathbb{R}^{3}$. Hence, \eqref{s1} and \eqref{s2} give that
\begin{align}\label{s3}
	b(t)-b(0) \leq  & C\sup _{s \in[0, T]}  \|(-\Delta)^{-1} \operatorname{div}(\rho u)(s) \|_{L^{\infty}} +C\int_{0}^{T}\|b\|_{L^{\infty}}^{2} \mathrm{~d} \tau\notag\\
	& +C\int_{0}^{T} \| [u_{i}, R_{i j} ] (\rho u_{j} )\|_{L^{\infty}} d \tau+C \int_{0}^{T} \|(-\Delta)^{-1} \operatorname{div} \operatorname{div}(b \otimes b) \|_{L^{\infty}} d \tau
	\triangleq\sum_{i=1}^{4} N_{i}.
\end{align}

For $N_{1}$, \eqref{xu0} implies that
\begin{align}\label{xu3}
N_{1} \leq C \bar{\rho}^{\frac{3}{4}}\|\sqrt{\rho} u\|_{L^{2}}^{\frac{1}{2}}\|\nabla u\|_{L^{2}}^{\frac{1}{2}}.
\end{align}
Using Gagliardo--Nirenberg inequality, H\"{o}lder's inequality, the commutator estimates,  \eqref{c9}, and \eqref{c3}, we have
\begin{align}
	N_{2} &\leq C  \int_{0}^{T} \|b\|_{L^{\infty}}^{2} d \tau\notag\\
	      &\leq C \int_{0}^{T} \bar{\rho}^{\frac{1}{2}} \|\nabla b\|_{L^{2}}  \bar{\rho}^{-\frac{1}{2}}\|\nabla^{2} b \|_{L^{2}} d \tau\notag\\
	      &\leq  C \bigg(\bar{\rho} \int_{0}^{T}\|\nabla b\|_{L^{2}}^{2} d \tau\bigg)^{\frac{1}{2}}\bigg(\int_{0}^{T}\bar{\rho}^{-1}\|\nabla^{2} b\|_{L^{2}}^{2} d \tau\bigg)^{\frac{1}{2}}\notag\\
	      &\leq C \bigg[\bigg(1+\bar{\rho}+\frac{1}{\bar{\rho}}\bigg) S_{0}^{\prime} S_{0}^{\prime\prime}\bigg]^{\frac{1}{2}},\label{xu9}\\
N_{3}&\leq C\int_{0}^{T} \| [u^{i}, R_{i j} ] (\rho u^{j} ) \|_{L^{\infty}} d \tau\notag\\
	 &\leq C\int_{0}^{T}\|[u, R_{i j}](\rho u)\|_{L^{3}}^{\frac{1}{5}}\|\nabla[u, R_{i j}](\rho u)\|_{L^{4}}^{\frac{4}{5}}d \tau \notag\\
	 &\leq C \int_{0}^{T} \bar{\rho}\|\nabla u\|_{L^{2}}\|\nabla u\|_{L^{6}} d \tau \notag\\
	 &\leq  C \int_{0}^{T} \bar{\rho}\|\nabla u\|_{L^{2}}\Big(\bar{\rho}^{\frac{1}{2}}\|\sqrt{\rho} \dot{u}\|_{L^{2}}+\bar{\rho}\|\nabla \theta\|_{L^{2}}+\||b||\nabla b|\|_{L^{2}}\Big) d \tau \notag\\
	 &\leq   C \bar{\rho}\bigg(\int_{0}^{T}\|\nabla u\|_{L^{2}}^{2} d \tau\bigg)^{\frac{1}{2}}\bigg(\int_{0}^{T} \bar{\rho}\|\sqrt{\rho} \dot{u}\|_{L^{2}}^{2} d \tau\bigg)^{\frac{1}{2}} \notag\\
	 &\quad+ C \bar{\rho}^{2}\bigg(\int_{0}^{T}\|\nabla u\|_{L^{2}}^{2} d \tau\bigg)^{\frac{1}{2}}\bigg(\int_{0}^{T}\|\nabla \theta\|_{L^{2}}^{2} d \tau\bigg)^{\frac{1}{2}}+ C \bar{\rho}\bigg(\int_{0}^{T}\|\nabla u\|_{L^{2}}^{2} d \tau\bigg)^{\frac{1}{2}}\bigg(\int_{0}^{T}\||b|| \nabla b |\|_{L^{2}}^{2} d \tau\bigg)^{\frac{1}{2}}\notag\\
	 &\leq C\bigg[\bigg(1+\bar{\rho}+\frac{1}{\bar{\rho}}\bigg) S_{0}^{\prime} S_{0}^{\prime\prime}\bigg]^{\frac{1}{2}}\label{xu4},\\
	N_{4}&\leq C \int_{0}^{T}  \|(-\Delta)^{-1} \operatorname{div} \operatorname{div}(b \otimes b) \|_{L^{\infty}}  d \tau     \notag\\
	&  \leq C \int_{0}^{T} \|(-\Delta)^{-1} \operatorname{div} \operatorname{div}(b \otimes b) \|_{L^{6}}^{\frac{1}{2}}  \|\nabla (-\Delta)^{-1} \operatorname{div} \operatorname{div}(b \otimes b) \|_{L^{6}}^{\frac{1}{2}}   d \tau\notag\\
	&\leq C \int_{0}^{T}\||b||\nabla b|\|_{L^{2}}^{\frac{1}{2}}\||b||\nabla b|\|_{L^{6}}^{\frac{1}{2}}  d \tau\notag\\
	&\leq C \int_{0}^{T}\|b\|_{L^{\infty}}\|\nabla b\|_{L^{2}}^{\frac{1}{2}}\|\nabla^{2} b\|_{L^{2}}^{\frac{1}{2}}  d \tau\notag\\
	&\leq C \bigg[\bigg(1+\bar{\rho}+\frac{1}{\bar{\rho}}\bigg) S_{0}^{\prime} S_{0}^{\prime\prime}\bigg]^{\frac{1}{2}}\label{xu5}.
\end{align}

Integrating \eqref{xu1} over $[0, t]$ and using \eqref{s3}--\eqref{xu5}, one derives that
\begin{align}\label{s4}
	\ln \rho^{\delta}(x, s) \leq & \ln (\bar{\rho}+\delta)+b(t)-b(0) \notag\\
	\leq & \ln (\bar{\rho}+\delta)+ C \bar{\rho}^{\frac{3}{4}}\|\sqrt{\rho} u\|_{L^{2}}^{\frac{1}{2}}\|\nabla u\|_{L^{2}}^{\frac{1}{2}}+C \bigg[\bigg(1+\bar{\rho}+\frac{1}{\bar{\rho}}\bigg) S_{0}^{\prime} S_{0}^{\prime\prime}\bigg]^{\frac{1}{2}}\notag\\
	\leq& \ln (\bar{\rho}+\delta) +\ln \frac{3}{2},
\end{align}
where the smallness of $\epsilon$ in \eqref{zui} has been used in the last step. Letting $\delta\to 0$ in \eqref{s4} yields that
\begin{align*}
\rho \leq \frac{3}{2} \bar{\rho}.
\end{align*}
for all $(x,t)\in\mathbb{R}^{3}\times[0,T]$. This completes the proof of Corollary \ref{co2}.
\end{proof}

\begin{proof}[Proof of Theorem \ref{thm1}.]
Recall that a Serrin-type blow-up criterion for strong solutions obtained in \cite{HL13}. More precisely,
\begin{align}\label{blow}
\lim _{T \rightarrow T^{*}}\left(\|\rho\|_{L^{\infty}\left(0, T ; L^{\infty}\right)}+\|u\|_{L^{s}\left(0, T ; L^{r}\right)}\right)=\infty,
\quad \frac{2}{s}+\frac{3}{r} \leq 1
\end{align}
with exponents $r$ and $s$ satisfying $s>1$, $3<r \leq \infty$, where $T^{*}$ denotes the maximal existence time of the strong solution.

According to the local existence theory in Lemma \ref{local}, there exists a unique local strong solution $(\rho, u, \theta, b)$ to the problem \eqref{1}--\eqref{5}. Denote by $T_{\max}>0$ its maximal existence time. Combining Proposition \ref{pro} with the blow-up criterion \eqref{blow} leads to the global existence and uniqueness of strong solutions. Indeed, the uniform upper bound for $\rho$ has been obtained in Corollary \ref{co2}. Moreover, taking $s=4$ and $r=6$ in \eqref{blow} and using \eqref{c3} and \eqref{c9}, there exists a constant $\bar{C}>0$ independent of $T_{\max}$ such that
\begin{align}\label{blow1}
\sup _{t \in[0, T_{\max}]} \|\rho\|_{L^{\infty}}+\int_{0}^{T_{\max }}\|u\|_{L^{6}}^{4}dt \leq \sup _{t \in[0, T_{\max}]} \|\rho\|_{L^{\infty}}+C\sup _{t \in[0, T_{\max}]}\|\nabla u(t)\|_{L^{2}}^{2} \int_{0}^{T_{\max }}\|\nabla u\|_{L^{2}}^{2}dt \leq \bar{C}.
\end{align}
Thus, \eqref{blow} and \eqref{blow1} imply that $T_{\max }=\infty$. Consequently, we can extend the local strong solution to be a global one.
\end{proof}

\section*{Conflict of interest}
The authors have no conflicts to disclose.

\section*{Data availability}
No data was used for the research described in the article.

\end{document}